\documentclass[12pt,a4paper,reqno]{amsart}
\usepackage[utf8]{inputenc}
\usepackage[T1]{fontenc}
\usepackage{indentfirst}
\usepackage{textcomp,lscape,rotating}
\usepackage[english]{babel}
\usepackage{enumitem}
\usepackage{amsmath,amsfonts,amssymb,amsopn,amscd,amsthm}
\usepackage{mathrsfs,shuffle,bbm}
\usepackage{stmaryrd}
\usepackage{mathpazo}
\usepackage{graphicx}
\usepackage[dvipsnames]{xcolor}
\usepackage[colorlinks=true,citecolor=DarkOrchid,linkcolor=NavyBlue]{hyperref}
\usepackage{tikz}
\usepackage{ytableau}
\usepackage[scale=0.81, centering]{geometry}

    \newtheorem{definition}{Definition}
    \newtheorem{lemma}[definition]{Lemma}
    \newtheorem{assumption}[definition]{Assumption}
    \newtheorem{theorem}[definition]{Theorem}
    \newtheorem{proposition}[definition]{Proposition}
    \newtheorem{corollary}[definition]{Corollary}
    \theoremstyle{remark}
    \newtheorem{example}[definition]{Example}
    \newtheorem{remark}[definition]{Remark}

    \makeatletter
    \def\thm@space@setup{\thm@preskip=0.5cm   \thm@postskip=0.5cm}
    \makeatother

    \newcommand{\N}{\mathbb{N}}
    
    \newcommand{\R}{\mathbb{R}}
    \newcommand{\C}{\mathbb{C}}
    \newcommand{\E}{\mathrm{e}}
    \newcommand{\I}{\mathrm{i}}
    \newcommand{\esper}{\mathbb{E}}
    \newcommand{\proba}{\mathbb{P}}

    \newcommand{\leb}{\mathrm{L}}
    \newcommand{\dkol}{d_\mathrm{Kol}}
    \newcommand{\testf}{\mathscr{T}}
    \newcommand{\card}{\mathrm{card}}
    \newcommand{\XEC}{\mathbf{X}}
    \newcommand{\YEC}{\mathbf{Y}}
    \newcommand{\XIEC}{\boldsymbol{\xi}}
    \newcommand{\hypergeom}[3]{{}_2F_1\binom{{#1},{#2}}{#3}}
    \newcommand{\eps}{\varepsilon}
    \newcommand{\spa}{\mathfrak{X}}

    \newcommand{\gauss}{\mathcal{N}_{\R}(0,1)}
    \newcommand{\lle}{\left[\!\left[} 
    \newcommand{\rre}{\right]\!\right]} 
    \newcommand{\scal}[2]{\left\langle #1\vphantom{#2}\,\right |\left.#2 \vphantom{#1}\right\rangle}
    \newcommand{\DD}[1]{\,d\hspace*{-0.3mm}{#1}}

    \DeclareMathOperator{\Var}{Var}

\setlist[enumerate]{itemsep=10pt,topsep=10pt}
\setlist[itemize]{itemsep=5pt,topsep=5pt}

\title{Local limit theorems and mod-$\phi$ convergence}
\author{Martina Dal Borgo \and Pierre-Loïc Méliot \and Ashkan Nikeghbali}

\begin{document}

\begin{abstract}
We prove local limit theorems for mod-$\phi$ convergent sequences of random variables, $\phi$ being a stable distribution. In particular, we give two new proofs of the local limit theorem stated in \cite{DKN15}: one proof based on the notion of \emph{zone of control} introduced in \cite{FMN17}, and one proof based on the notion of \emph{mod-$\phi$ convergence in $\leb^1(\I\R)$}. These new approaches allow us to identify the infinitesimal scales at which the stable approximation is valid. We complete our analysis with a large variety of examples to which our results apply, and which stem from random matrix theory, number theory, combinatorics or statistical mechanics.
\end{abstract}

\maketitle

\tableofcontents

\clearpage

\section{Introduction}

Let $(Y_n)_{n \in \N}$ be a sequence of random variables which admit a limit in law $Y_\infty$ as $n$ goes to infinity. We assume that the distribution of $Y_\infty$ is absolutely continuous with respect to the Lebesgue measure; thus, there is a density $p(x)$ such that 
$$\proba[Y_\infty\in (a,b)] = \int_a^b p(x)\DD{x}.$$
The convergence in law $Y_n \rightharpoonup Y_\infty$ amounts then to
$$\lim_{n \to \infty} \proba[Y_n\in (a,b)] = \int_a^b p(x)\DD{x}.$$
for any $a<b$. In this setting, a \emph{local limit theorem} for the sequence $(Y_n)_{n \in \N}$ is a statement of the following form: for some sequence $(s_n)_{n \in \N}$ going to $+\infty$, and any $x,a,b \in \R$,
\begin{equation}
\lim_{n \to \infty} s_n\,\,\proba\!\left[Y_n-x \in \frac{1}{s_n}(a,b)\right] =  p(x)\,(b-a).\label{eq:locallimit1}
\end{equation}
Thus, we are interested in the asymptotics of the probability for $Y_n$ to fall in a small interval of size $(s_n)^{-1}$. More generally, given a bounded measurable subset $B$ whose boundary $\partial B$ has Lebesgue measure $m(\partial B)=0$, we want to prove that for some scales $s_n \to +\infty$,
\begin{equation}
\lim_{n \to \infty} s_n\,\,\proba\!\left[Y_n-x \in \frac{1}{s_n}\,B\right] =  p(x)\,m(B).\label{eq:locallimit2}
\end{equation}
Notice that the convergence in law $Y_n \rightharpoonup Y_\infty$ does not imply this kind of result. Besides, for many convergent sequences $(Y_n)_{n\in \N}$, there exist some scales $(s_n)_{n\in \N}$ for which the probability on the left-hand side of Equation \eqref{eq:locallimit2} cannot be equivalent to $p(x)\,m(B)/s_n$. For instance, if $Y_n = \frac{N_n}{s_n}$ is a renormalisation of an \emph{integer-valued} statistics $N_n$, then at the scale $s_n$, Equation \eqref{eq:locallimit1} cannot be true, because $a,b \mapsto \proba[N_n \in (a,b)]$ is not continuous in $a$ and $b$. The goal of this paper is to show that in the setting of mod-$\phi$ convergent sequences, there is a large range of scales $(s_n)_{n \in \N}$ for which the local limit theorem is satisfied.

\subsection{Mod-\texorpdfstring{$\phi$}{phi} convergence}
We start by recalling the notion of mod-$\phi$ convergent sequences, which has been introduced in \cite{JKN11,DKN15,FMN16}.

\begin{definition}
Let $D\subseteq \C$ be a subset of the complex plane containing $0$, and $(X_n)_{n\in\N}$ be a sequence of real-valued random variables whose moment generating functions $\esper[\E^{zX_n}]$ are well defined over $D$. We also fix a non-constant infinitely divisible distribution $\phi$ whose Laplace transform is well defined over $D$ and has L\'evy exponent $\eta$:
\begin{equation*}
\forall z\in D,\qquad \int_{\R}\E^{zx}\,\phi(\!\DD{x})=\E^{\eta(z)}.
\end{equation*}
We then say that $(X_n)_{n\in\N}$ converges mod-$\phi$ over $D$ with parameters $(t_n)_{n \in \N}$ and limiting function $\psi:D\rightarrow \C$ if, locally uniformly on $D$,
\begin{equation*}
\psi_n(z):=\esper\!\left[\E^{zX_n}\right]\,\E^{-t_n\eta(z)}\longrightarrow\psi(z).
\end{equation*}
Here, $(t_n)_{n\in\N}$ is some deterministic sequence of positive numbers with $\lim_{n \to \infty} t_n =+\infty$, and  $\psi(z)$ is a continuous function on $D$ such that $\psi(0)=1$. 
\end{definition}

Let us comment a bit this definition with respect to the choice of the domain. If $D=\I\R$, then we are looking at Fourier transforms, or ratios thereof. So, there is no problem of definition of these quantities (recall that the Fourier transform $\E^{\eta(\I \xi)}$ of an infinitely divisible distribution does not vanish, see \cite[Lemma 7.5]{Sato99}). We shall then simply speak of mod-$\phi$ convergence (or mod-$\phi$ convergence in the Fourier sense if we want to be precise), and denote
$$\theta_n(\xi):= \esper\!\left[\E^{\I \xi X_n}\right]\,\E^{-t_n\eta(\I \xi)}$$
and $\theta(\xi) = \lim_{n \to \infty} \theta_n(\xi)$. On the other hand, if $D$ is an open subset containing $0$, then the Laplace transforms must be analytic functions on this domain. Mod-$\phi$ convergence on such a domain, or even on the whole complex plane $\C$ occurs often when the reference law $\phi$ is the standard Gaussian distribution, with density 
$$\frac{1}{\sqrt{2\pi}}\,\E^{-\frac{x^2}{2}}$$
and with Lévy exponent $\eta(z)=\frac{z^2}{2}$. In this paper, we shall consider domains of convergence $D$ equal to $\I \R$, or $\C$, or $\R$ (in Section \ref{subsec:curieweiss}). In some cases, we shall also require that the local uniform convergence of the residues $\psi_n \to \psi$ or $\theta_n \to \theta$ occurs in $\leb^1(D)$; see Definitions \ref{def:L1modphi} and \ref{def:L1R}. 
\medskip

\subsection{Stable distributions on the real line}
In this article, we shall only be interested in the case where $\phi$ is a stable distribution:
\begin{definition}
Let $c>0$, $\alpha \in (0,2]$ and $\beta \in [-1,1]$. 
The stable distribution of parameters $(c,\alpha,\beta)$ is the infinite divisible law $\phi=\phi_{c,\alpha,\beta}$ whose Fourier transform $\int_{\R}\E^{\I \xi x}\phi(dx)=\E^{\eta(\I \xi)}$
has L\'evy exponent $\eta=\eta_{c,\alpha,\beta}$  given by
$$
\eta_{c,\alpha,\beta}(\I\xi)=-\left|c\xi\right|^\alpha\left(1-\I\beta h(\alpha,\xi)\,\mathrm{sgn}(\xi)\right),
$$
where $\mathrm{sgn}(\xi)$ is the sign of $\xi$, and
\begin{equation*}
 h(\alpha,\xi)=\begin{cases}
 \tan\left(\frac{\pi \alpha}{2}\right),\qquad \text{if}\ \alpha\not =1,\\
 -\frac{2}{\pi}\log\left|\xi\right|,\quad\ \text{if}\ \alpha=1.
 \end{cases}
\end{equation*}
\end{definition}

We refer to \cite[Chapter 3]{Sato99} for details on stable distributions, see in particular Theorem 14.15 in \emph{loc.~cit.}~for the formula for the Lévy exponents. Using the above definition, one sees that the Lévy exponent of a stable law satisfies the scaling property:
\begin{equation*}
t\eta_{c,\alpha,\beta}\left(\frac{\I\xi}{t^{1/\alpha}}\right)=\begin{cases}
 \eta_{c,\alpha,\beta}(\I\xi)&\text{if } \alpha\neq 1,\\
 \eta_{c,\alpha,\beta}(\I\xi)-\left(\frac{2c\beta}{\pi}\log t\right)\I\xi &\text{if } \alpha=1.
 \end{cases}
\end{equation*}
Stable distributions include:
\begin{itemize}
    \item the standard Gaussian law, corresponding to the triplet $(c,\alpha,\beta)=\left (\frac{1}{\sqrt{2}},2,0\right)$;
    \item the standard Cauchy law, corresponding to the triplet  $(c,\alpha,\beta)=\left (1,1,0\right)$ ;
    \item the standard L\'evy law, corresponding to the triplet  $(c,\alpha,\beta)=\left (1,\frac{1}{2},1\right)$.
\end{itemize}
Note that, since $|\E^{\eta_{c,\alpha,\beta}(\I\xi)}|= \E^{-\left|c\xi\right|^\alpha}$ is integrable, any stable law  is absolutely continuous with the respect to the Lebesgue measure. Throughout this article, we denote by $p_{c,\alpha,\beta}(x)\DD{x}$ the density of the stable distribution $\phi_{c,\alpha,\beta}$. On the other hand, it is an easy exercise to show that the definition of mod-stable convergence together with the scaling property of the L\'evy exponent imply the following proposition (see \cite[Proposition 1.3]{FMN17}).

\begin{proposition}\label{prop:convlaw}
If $(X_n)_{n\in\N}$ converges mod-$\phi_{c,\alpha,\beta}$ (in the Fourier sense on $D=\I \R$) with parameters $(t_n)_{n \in \N}$ and limiting function $\theta$, then
\begin{equation*}
Y_n:=\begin{cases}
\frac{X_n}{(t_n)^{1/\alpha}}& \text{if } \alpha\neq 1,\\
\frac{X_n}{t_n}-\frac{2c\beta}{\pi}\log t_n&\text{if } \alpha=1
\end{cases}
\end{equation*}
converges in law towards $\phi_{c,\alpha,\beta}$.
\end{proposition}
\medskip

The stable laws are well known to be the attractors of the laws of sums of independent and identically distributed random variables. More precisely, fix a cumulative distribution function $F: \R \to [0,1]$, and consider a sequence $(X_n)_{n \in \N}$ of independent random variables with $\proba[X_n \leq x] = F(x)$ for any $n$. If the scaled sum 
$$\frac{S_n-A_n}{B_n} = \frac{X_1+X_2+\cdots+X_n - A_n}{B_n}$$
admits a limiting distribution for some choice of parameters $A_n$ and $B_n$, then this limiting distribution is necessarily a stable law $\phi_{c,\alpha,\beta}$ (up to a translation if $A_n$ is not chosen correctly); see \cite[Chapter 7]{GK68}. One then says that $F$ belongs to the domain of attraction of the stable law of parameter $(\alpha,\beta)$ ($c$ can be chosen by changing $B_n$). Necessary and sufficient conditions on $F$ for belonging to the domain of attraction of a stable distribution $\phi_{c,\alpha,\beta}$ are given in \cite[Chapter 7, \S35]{GK68} and \cite[Chapter 2, \S6]{IL71}. In terms of Fourier transforms, one criterion is the following \cite[Theorem 2.6.5]{IL71}: a probability measure $\mu$ belongs to the domain of attraction of a stable law of parameter $(\alpha,\beta)$ if and only if its Fourier transform writes in the neighborhood of the origin as
$$\widehat{\mu}(\xi) = \E^{\I m \xi - |c\xi|^\alpha (1-\I \beta h(\alpha,\xi)\,\mathrm{sgn}(\xi))\,s(\xi)\,(1+\eps(\xi))},$$
where $\lim_{\xi \to 0}\eps(\xi)=0$, and $\xi \mapsto s(\xi)$ is a slowly varying function at $0$ in the sense of Karamata (see \cite{BGT87}), meaning that $s$ is positive and
$$\forall a \neq 0,\,\,\,\lim_{\xi \to 0} \frac{s(a\xi)}{s(\xi)} = 1.$$
In this representation, $s(\xi)$ is positive real valued, whereas the function $\eps(\xi)$ can be complex.
The central limit theorem for random variables in the domain of attraction of a stable law is completed by a local limit theorem \cite{Shepp64,Stone65,Feller67}. In this paper, we shall revisit this theorem by showing that it is a simple consequence of a result of approximation by smooth test functions (Section \ref{subsec:stonefeller}).
\medskip

\subsection{Main results and outline of the paper}

We fix a stable law $\phi_{c,\alpha,\beta}$, and we consider a sequence $(X_n)_{n \in \N}$ of random variables that is mod-$\phi_{c,\alpha,\beta}$ convergent over some domain $D$, with parameters $(t_n)_{n \in \N}$ and limiting function $\psi$. The goal of this paper is to show that the central limit theorem of Proposition \ref{prop:convlaw} goes together with a local limit theorem
\begin{equation*}
\lim_{n\to\infty}(s_n)\,\proba[Y_n - x\in (s_n)^{-1}B] =p_{c,\alpha,\beta}(x)\,m(B),
\end{equation*}
at scales $(s_n)_{n\in \N}$ that are determined by the quality of the convergence of the residues $\theta_n \to \theta$ (or $\psi_n \to \psi$ if $D = \C$ or $D=\R$). To be more precise, with $(X_n)_{n \in \N}$ and $(Y_n)_{n\in \N}$ as in Proposition \ref{prop:convlaw}, there are two possible situations:
\begin{enumerate}
\item There exists a minimal scale $s_n \to +\infty$ such that:
\begin{itemize}
     \item if $s_n' =o(s_n)$ and $s_n'\to+\infty$, then $\proba[Y_n-x \in (s_n')^{-1}B]$ can be approximated by the probability for the stable reference law;
     \item on the contrary, $\proba[Y_n-x \in (s_n)^{-1}B]$ cannot be approximated by the stable distribution, because of combinatorial or arithmetic properties of the underlying random model (for instance, if $Y_n$ comes from a lattice-valued random variable).
\end{itemize}
In this case, the theory of zones of control introduced in \cite{FMN17} will enable us to determine the scale $s_n$, or at least a sequence $r_n =O(s_n)$ up to which the stable approximation holds. This is the content of Theorem \ref{thm:locallimit1}.
\item The stable approximation of the probability $\proba[Y_n-y \in (s_n)^{-1}B]$ holds as soon as $s_n \to +\infty$. Hence, for \emph{any} infinitesimal scale $\eps_n \to 0$, the probability of $Y_n$ falling in an interval with this scale $\eps_n$ is asymptotically equivalent to the probability given by the stable reference law. Theorem \ref{thm:locallimit2} gives a sufficient condition which relies on the notion of mod-$\phi$ convergence in $\leb^1(\I\R)$. Note that this cannot occur if the $X_n$'s are lattice valued.
\end{enumerate}

Both approaches extend the results of the paper \cite{DKN15}, which gave a set of conditions (\emph{cf.}~the hypotheses H1-H3 in \emph{loc.~cit.}) that implied a local limit theorem with respect to a symmetric stable law ($\beta=0$). In many cases, the results of our paper improve the range of validity of this local limit theorem. On the other hand, both approaches rely on estimates for the differences
$$\esper[g_n(Y_n)]-\esper[g_n(Y)],$$
where $Y\sim \phi_{c,\alpha,\beta}$ and where the $g_n$'s are integrable test functions whose Fourier transforms have compact support. These test functions already played an essential role in \cite{FMN17} when studying the speed of convergence in Proposition \ref{prop:convlaw}; we recall their main properties in Section \ref{sec:testfunction}. Note that recently, other techniques have been developed in order to obtain local limit theorems: for integer-valued random variables, let us mention in particular the use of Landau--Kolmogorov type inequalities \cite{RR15}, the use of translated Poisson random variables instead of normal variables, and estimates coming from Stein's method \cite{BRR17}.
\medskip

Let us now detail the content of the paper. The theoretical results are given in Sections \ref{subsec:compactlysupportedfourier}, \ref{sec:zone} and \ref{subsec:l1}. The other sections are devoted to a large variety of examples and applications:
\begin{itemize}
    \item In Section \ref{subsec:stonefeller}, we first look at sequences $S_n=\sum_{i=1}^nA_{i}$, where the $A_i$'s are i.i.d.~random variables. In this case, we show that our results on test functions imply the well-known local limit theorems for distributions in the domain of attraction of a stable law. Namely, we recover the theorem of Shepp \cite{Shepp64} for laws with finite variance, and the generalizations of Stone and Feller \cite{Stone65,Feller67} for the cases $\alpha\in (0,2)$.


    \item Section \ref{sec:sum} is devoted to the analysis of sums of random variables which are not identically distributed, or not independent. We start with random variables that can be represented in law by sums or series of independent random variables: 
    \begin{itemize}
        \item the size of a random integer partition or a random plane partition chosen with probability proportional to $q^{|\lambda|}$ (Section \ref{subsec:partition});
        \item the number of zeroes of a random analytic series in a disc around the origin, and more generally the number of points of a determinantal point process that fall in a compact subset (Section \ref{subsec:zeroes});
        \item the random zeta functions $$\log\left(\prod_{p\leq N} \frac{1}{1-\frac{X_p}{\sqrt{p}}}\right),$$
         where the $X_p$'s are labeled by prime numbers, and are independent uniform random variables on the unit circle (Section \ref{subsec:randomzeta}).
    \end{itemize}
    The first example was kindly suggested to us by A.~Borodin. On the other hand, the random zeta functions have already been studied in \cite[Section 3, Example 2]{KN12} and \cite[Section 3.5]{DKN15}, in connection to Ramachandra's conjecture on the denseness in the complex plane of the values of the Riemann $\zeta$ function on the critical line. For these three examples, we establish the mod-Gaussian convergence with an adequate zone of control, and we deduce from it a local limit theorem.

    \item More generally, we can work with sums $S_n = \sum_{v\in V_n} A_v$ of random variables which have a dependency structure encoded in a dependency graph or in a weight\-ed dependency graph. In \cite[Theorem 9.1.8]{FMN16} and \cite[Proposition 5.3]{FMN17}, we proved that these hypotheses imply uniform bounds on the cumulants of $S_n$. From these bounds, it is usually possible to establish a zone of control for a renormalization of $(S_n)_{n\in\N}$, and the local limit theorem is then a straightforward application of Theorem \ref{thm:locallimit1}. In Sections \ref{subsec:dependencygraph} and \ref{subsec:markov}, we study in particular the subgraph counts in random Erdös--Rényi graphs, and the number of visits of a finite ergodic Markov chain.

    \item The next applications (Section \ref{sec:matrix}) are based on the results in \cite{DHR18,DHR19}, where mod-Gauss\-ian convergence has been proven for sequences stemming either from random matrices or from the Coulomb gas context. For all these examples, one can compute a large zone of control, which combined with our main result (Theorem \ref{thm:locallimit1}) provides a local limit theorem. The precise models that we shall study are:
    \begin{itemize}
        \item in Section \ref{subsec:charge}, the charge ensembles proposed in \cite{RSX13,SS14}, which consist of charge $1$ and charge $2$ particles located on the real line or the circle and interacting via their pairwise logarithmic repulsion, and with an harmonic attraction towards the origin in the real case. In the regime where the two types of particles have the same magnitude, the asymptotic behavior of the number of particles with charge $1$ was studied by Dal Borgo, Hovhannisyan and Rouault in \cite{DHR18}.
        \item in Section \ref{subsec:gue}, the logarithm of the determinant of a random matrix of the Gaussian Unitary Ensemble. The central limit theorem for this quantity was shown by Delannay and Le Ca\"er in \cite{DLC00}, and moderate deviations and Berry--Esseen bounds were established by D\"oring and Eichelsbacher in \cite{DE13}.
        \item in Section \ref{subsec:beta}, the logarithm of the determinant of a random matrix of the $\beta$-Laguerre, the $\beta$-Jacobi and the $\beta$-uniform Gram ensemble, for general $\beta>0$. The corresponding central limit theorems were established by Rouault in \cite{Rou07}.
        \item last, in Section \ref{subsec:circular}, the logarithm of the characteristic polynomial of a random matrix of the $\beta$-circular Jacobi ensemble, for general $\beta>0$. An asymptotic study of these quantities relying on the theory of deformed Verblunsky coefficients was proposed by Bourgade, Nikeghbali and Rouault in \cite[Section 5]{BNR09}.
    \end{itemize}
    To the best of our knowledge, the local limit theorems are new for all these examples. Using the polynomial structure of the partition function and applying an argument of Bender \cite[Theorem 2]{Bender}, Forrester gave in \cite[Section 7.10]{For10} a local limit theorem for a two-component Coulomb gas model on the circle with charge ratio $2:1$. This model is analogous to the ensemble proposed in \cite{SS14} and studied in our Section \ref{subsec:charge}, but different.

    \item In our last Section \ref{sec:l1mod}, we consider examples that correspond to the second case of the alternative previously described. In \cite{DKN15,FMN17}, it has been proved that the winding number of the planar Brownian motion starting at $z=1$ converges in the mod-Cauchy sense, with a large zone of control. We show in \S\ref{subsec:brownian} that we have in fact mod-Cauchy convergence in $\leb^1(\I \R)$, and therefore a local limit theorem that holds at any infinitesimal scale. On the other hand, in \cite[Section 3]{MN15}, a notion of mod-Gaussian convergence in $\leb^1(\R)$ was introduced, leading to a simple proof of classical results of Ellis and Newman \cite{EN78} on the magnetisation of the Curie--Weiss model at critical temperature. In Section \ref{subsec:curieweiss}, we extend and generalise \cite[Theorem 22]{MN15}, by proving a local limit theorem for this magnetisation, which holds for more scales than in \emph{loc.~cit.}
\end{itemize}
\bigskip

\section{Smooth test functions}\label{sec:testfunction}
In this section, we introduce the main tool for the proof of local limit theorems, namely, a space of smooth test functions $\testf_0(\R)$. This functional space already appeared in \cite{FMN17} when studying estimates of the speed of convergence in central limit theorems. We state in \S\ref{subsec:compactlysupportedfourier} an approximation lemma which will enable the proof of our local limit theorems, and in Section \ref{subsec:stonefeller}, we explain how to recover quickly the Stone--Feller local limit theorem by using the space of test functions.
\medskip

\subsection{Functions with compactly supported Fourier transforms}\label{subsec:compactlysupportedfourier}
In this section, all the spaces of functions are spaces of complex functions on the real line $f : \R \to \C$. We note
\begin{itemize}
     \item $\mathscr{C}^\infty(\R)$ (respectively, $\mathscr{D}(\R)$) the space of infinitely differentiable functions on $\R$ (respectively, infinitely differentiable and compactly supported).
     \item $\leb^1(\R)$ the space of measurable functions on $\R$ that are integrable with respect to the Lebesgue measure.
 \end{itemize}  
On the other hand, if $f \in \leb^1(\R)$, its Fourier transform is the continous function
$$\widehat{f}(\xi) = \int_{\R} f(x)\,\E^{\I x \xi}\DD{x}.$$

\begin{definition}
A smooth test function is a function $f \in \leb^1(\R)$ whose Fourier transform is compactly supported:
$$\exists K\geq 0,\,\,\,\forall \xi \notin [-K,K],\,\,\,\widehat{f}(\xi) =0.$$
\end{definition}
We denote by $\testf_0(\R)$ the space of smooth test functions; it is a subspace of $\mathscr{C}^\infty(\R)$, and if $f \in \testf_0(\R)$, then $f$ and all its derivatives tend to $0$ at infinity, and $f$ satisfies the Plancherel inversion formula
$$f(x) = \frac{1}{2\pi}\int_{-K}^K \widehat{f}(\xi)\,\E^{-\I \xi x}\DD{\xi},$$
where $[-K,K]$ is a support for $\widehat{f}$. We refer to \cite[Section 2.2]{FMN17} for details on the functional space $\testf_0(\R)$. An essential property of $\testf_0(\R)$ is the following approximation result, proven in \cite[Theorem 4]{DKN15}:

\begin{theorem}\label{thm:approx}
Let $f \in \mathscr{D}(\R)$. For any $\eta >0$, there exists two smooth test functions $g_1,g_2 \in \testf_0(\R)$ such that $g_1 \leq f \leq g_2$ and
$$\int_{\R} (g_2(x)-g_1(x)) \DD{x}\leq \eta.$$
\end{theorem}
\noindent Using approximations by smooth functions in $\mathscr{D}(\R)$, one can extend Theorem \ref{thm:approx} to other functions. More precisely:
\begin{corollary}\label{cor:approx}
Let $B$ be a bounded measurable subset of $\R$ such that $\partial B$ has zero Lebesgue measure. For any $\eta>0$, there exists two smooth test functions $g_1,g_2 \in \testf_0(\R)$ such that $g_1 \leq 1_B \leq g_2$ and
$$\int_{\R}(g_2(x)-g_1(x))\DD{x}\leq \eta.$$
\end{corollary}
\begin{proof}
The function $x \mapsto 1_B(x)$ is bounded, and since $m(\partial B)=0$, it is almost everywhere continuous. Therefore, it is Riemann integrable, and one can frame it between two step functions $f_1$ and $f_2$ (locally constant functions with a finite number of values). In turn, one can classically approximate these two step functions by smooth compactly supported functions in $\mathscr{D}(\R)$, and finally one can use Theorem \ref{thm:approx} to replace the smooth compactly supported functions in $\mathscr{D}(\R)$ by smooth test functions in $\testf_0(\R)$.
\end{proof}
In the sequel, a bounded measurable subset $B \subset \R$ whose boundary $\partial B$ has zero Lebesgue measure will be called a \emph{Jordan measurable subset}.
\medskip

\subsection{Stone--Feller local limit theorem}\label{subsec:stonefeller}
In the next section, we shall use the approximation theorem \ref{thm:approx} to prove local limit theorems in the mod-$\phi$ setting. As a warm-up, let us explain how to recover the Stone--Feller local limit theorem for sums of i.i.d.~random variables in the attraction domain of a stable law.
\begin{theorem}[Stone, Feller]
Let $\mu$ be a non-lattice distributed probability measure which is in the attraction domain of $\phi_{c,\alpha,\beta}$, and which has its Fourier transform that writes as
$$\widehat{\mu}(\xi) = \E^{\I m \xi - |c\xi|^\alpha (1-\I \beta h(\alpha,\xi)\,\mathrm{sgn}(\xi))\,s(\xi)\,(1+\eps(\xi))},$$
with $s$ slowly varying at $0$ and $\lim_{\xi \to 0} \eps(\xi) = 0$. We assume to simplify that we are not in the case $\alpha = 1,\beta \neq 0$. We consider a sum $S_n = X_1+\cdots+X_n$ of i.i.d.~random variables with law $\mu$, and we define $A_n$ and $B_n$ by
$$A_n = nm  \qquad;\qquad (B_n)^\alpha = n\,s\left(\frac{1}{B_n}\right).$$
Assume that $(B_n)_{n \in \N}$ goes to $+\infty$. Then, for any $x \in \R$ and any Jordan measurable subset $C$ with $m(C)>0$,
$$\lim_{n \to \infty} B_n\,\proba\!\left[S_n \in A_n+B_nx+C\right] = p_{c,\alpha,\beta}(x)\,m(C) ,$$
where $p_{c,\alpha,\beta}(x) $ is the density at $x$ of the stable law $\phi_{c,\alpha,\beta}$.
\end{theorem}

\begin{remark}
The assumption $B_n \to + \infty$ is in fact a consequence of the other hypotheses, see \cite[Lemma in \S29]{GK68}; besides, in practice one can usually compute $B_n$ or an estimate of it. On the other hand, under the assumptions of the theorem, for any $\xi$ fixed in $\R$,
\begin{align*}
\esper[\E^{\I \xi \frac{S_n-A_n}{B_n}}] &= \exp\left(-|c\xi|^\alpha (1-\I \beta\, h(\alpha,\xi)\,\mathrm{sgn}(\xi))\,\frac{s\!\left(\frac{\xi}{B_n}\right)}{s\!\left(\frac{1}{B_n}\right)}\left(1+\eps\!\left(\frac{\xi}{B_n}\right)\right)\right) \\
&\to_{n \to \infty} \E^{-|c\xi|^\alpha (1-\I \beta h(\alpha,\xi)\,\mathrm{sgn}(\xi))}
\end{align*}
so we have the central limit theorem $\frac{S_n-A_n}{B_n} \rightharpoonup_{n \to \infty} \phi_{c,\alpha,\beta}$. The Stone--Feller theorem is a local version of this limiting result.
\end{remark}

\begin{proof}
If $Y_n = S_n-A_n-B_nx$, then we are interested in the asymptotics of the quantity $\proba[Y_n \in C]=\esper[1_{C}(Y_n)]$. By Corollary \ref{cor:approx}, it suffices to prove that for any $f \in \testf_0(\R)$, 
\begin{equation}
     \lim_{n\to \infty} B_n\,\esper[f(Y_n)] = \left(\int_{\R} f(y) \DD{y}\right)\,p_{c,\alpha,\beta}(x) \label{eq:stonefeller}
 \end{equation} 
for $f \in \testf_0(\R)$; the same result will then hold for $f=1_{C}$, hence the theorem.  We compute the left-hand side of \eqref{eq:stonefeller}, denoting $[-K,K]$ a support for $\widehat{f}$:
\begin{align*}
B_n\,\esper[f(Y_n)] &= \frac{B_n}{2\pi}\int_{\R} \widehat{f}(\xi)\,(\widehat{\mu}(-\xi))^n\, \E^{A_n\I\xi} \E^{B_n\I \xi x}\DD{\xi} \\
 &= \frac{B_n}{2\pi}\int_{-K}^K \widehat{f}(\xi)\,\E^{-n|c\xi|^\alpha (1-\I\beta h(\alpha,-\xi)\,\mathrm{sgn}(-\xi))\,s(-\xi)\,(1+\eps(-\xi))} \E^{ (A_n -nm)\I \xi +B_n\I \xi x}\DD{\xi} \\
 &= \frac{1}{2\pi}\int_{-KB_n}^{KB_n} \widehat{f}\left(-\frac{t}{B_n}\right)\,\E^{-|ct|^\alpha (1-\I\beta h(\alpha,t)\,\mathrm{sgn}(t))\,\frac{s(\frac{t}{B_n})}{s(\frac{1}{B_n})}\,(1+\eps(\frac{t}{B_n}))}\E^{-\I x t}\DD{t}.
 \end{align*}
Since $s$ is slowly varying around $0$, the pointwise limit as $n$ goes to infinity of the function in the integral is
$$\widehat{f}(0)\,\E^{-|ct|^\alpha (1-\I\beta h(\alpha,t)\,\mathrm{sgn}(t))}\,\E^{-\I x t}.$$
Let us explain why we can use the dominated convergence theorem. As $\mu$ is non-lattice distributed, the function $\widehat{\mu}(\xi)$ has modulus $1$ only for $\xi =0$, and therefore, the real part of the function $\xi \mapsto (1-\I \beta h(\alpha,\xi)\,\mathrm{sgn}(\xi))(1+\eps(\xi))$ does not vanish on $\R$. In particular, this real part stays bounded from below by a positive constant $C_1$ on some interval $[-K,K]$. On the other hand, in order to evaluate the ratio $s(\frac{t}{B_n})/s(\frac{1}{B_n})$, we use Karamata's representation theorem, which states that for $\xi \leq K$,
\begin{equation}
     s(\xi) = \exp\left(\eta_1(\xi) + \int_\xi^K \eta_2(u)\,\frac{\!\DD{u}}{u}\right) \label{eq:karamata}
 \end{equation}
with $\eta_1$ bounded measurable function admitting a limit $\eta_1(0)$ when $\xi \to 0$, and $\eta_2$ is a bounded measurable function with $\lim_{\xi \to 0} \eta_2(\xi)=0$; see \cite[Section 1.3]{BGT87}. Up to a modification of the pair $(\eta_1,\eta_2)$, we can assume that $|\eta_2|\leq \frac{\alpha}{2}$ for any $\xi \leq K$. Then, 
\begin{align*}
\frac{s\!\left(\frac{t}{B_n}\right)}{s\!\left(\frac{1}{B_n}\right)} &= \exp\!\left( \eta_1\!\left(\frac{t}{B_n} \right)-\eta_1\!\left(\frac{1}{B_n} \right)+\int_{\frac{t}{B_n}}^{\frac{1}{B_n}} \eta_2(u)\,\frac{\!\DD{u}}{u}\right)
\\
&\geq \exp\!\left(O(1) - \frac{\alpha}{2}\log t\right) \geq \frac{C_2}{t^\frac{\alpha}{2}}
\end{align*}
for some constant $C_2>0$, and any $t \leq KB_n$. Therefore, on the zone of integration,
$$\mathrm{Re}\left(|t|^\alpha (1-\I\beta h(\alpha,t)\,\mathrm{sgn}(t))\,\frac{s(\frac{t}{B_n})}{s(\frac{1}{B_n})}\left(1+\eps\!\left(\frac{t}{B_n}\right)\right)\right) \geq C_1C_2 |t|^{\frac{\alpha}{2}}.$$
This lower bound allows one to use the dominated convergence theorem, which shows that:
\begin{align*}
\lim_{n \to \infty} B_n \,\esper[f(Y_n)] &= \frac{1}{2\pi} \int_\R \widehat{f}(0)\,\E^{-|ct|^\alpha (1-\I\beta h(\alpha,t)\,\mathrm{sgn}(t))}\,\E^{-\I xt}\DD{t} \\
&= \left(\int_{\R} f(y) \DD{y}\right)\,p_{c,\alpha,\beta}(x).
\qedhere
\end{align*}
\end{proof}
\medskip

The proof adapts readily to the case $\alpha=1$, up to a modification of the parameters $A_n$ and $B_n$ when $\beta\neq 0$. Notice that our result of approximation by smooth test functions has reduced the notoriously difficult proof of the local limit theorem of Stone and Feller (due to Shepp in the case $\alpha=2$, for random variables with finite variance) to an application of Parseval's formula and of Karamata's representation theorem.
\bigskip

\section{Zones of control and local limit theorems}\label{sec:zone}
In this section, we explain how to obtain local limit theorems in the setting of mod-stable convergent sequences. The main idea is that the scales at which the stable approximation of a mod-$\phi$ convergent sequence is valid are dictated by:
\begin{enumerate}
    \item the behavior of the residues $\theta_n(\xi)$ and $\theta(\xi)$ around $0$;
    \item the maximal size of a zone on which the growth of these residues can be controlled.
\end{enumerate}
In the following, we fix a sequence of real-valued random variables $(X_n)_{n \in \N}$, a sequence of parameters $t_n \to +\infty$ and a reference stable law $\phi_{\alpha,\beta,c}$. In Section \ref{subsec:zone}, we recall the definition of zone of control, which is in some sense an improvement of the definition of mod-stable convergence. In Section \ref{subsec:llt_modstable}, we prove local limit theorems under this hypothesis of zone of control. 

\subsection{The notion of zone of control}\label{subsec:zone}
In \cite[Section 2.1]{FMN17}, the rate of convergence in Proposition \ref{prop:convlaw} was determined by using the notion of zone of control, which we recall here:
\begin{definition}\label{def:zone}
Let $(X_n)_{n \in \N}$ be a sequence of real-valued random variables, $\phi_{c,\alpha,\beta}$ be a stable law, and $t_n \to +\infty$. We set $\theta_n(\xi) = \esper[\E^{\I \xi X_n}]\,\E^{-t_n\eta_{c,\alpha,\beta}(\I \xi)}$. Consider the following assertions:
\begin{enumerate}[label=(Z\arabic*)]
\item\label{hyp:zone1} Fix $\nu>0,\ \omega>0$ and $\gamma\in\R$. There exists positive constants $K$, $K_1$ and $K_2$ that are independent of $n$ and such that, for all $\xi$ in the zone  $\left[-K\left(t_n\right)^\gamma,K\left(t_n\right)^\gamma\right]$,
$$
\left|\theta_n(\xi)-1\right|\le K_1\left|\xi\right|^\nu \exp\left(K_2\left|\xi\right|^\omega\right).
$$
\item\label{hyp:zone2} One has
$$
\alpha\le \omega, \qquad -\frac{1}{\alpha}< \gamma\le\frac{1}{\omega-\alpha},\qquad 0<K\le\left(\frac{c^\alpha}{2K_2}\right)^{\frac{1}{\omega-\alpha}}.
$$
\end{enumerate}
Note that if Condition \ref{hyp:zone1} holds for some parameters $\gamma>-\frac{1}{\alpha}$ and $\nu,\omega,K,K_1,K_2$, then \ref{hyp:zone2} can always be forced by increasing $\omega$, and then decreasing $K$ and $\gamma$. If Conditions \ref{hyp:zone1} and \ref{hyp:zone2} are satisfied, then we say that we have a zone of control $\left[-K\left(t_n\right)^\gamma,K\left(t_n\right)^\gamma\right]$ with index of control $(\nu,\omega)$.
\end{definition}

Let us make a few remarks on this definition. First, Conditions \ref{hyp:zone1} and \ref{hyp:zone2} imply that if $(Y_n)_{n \in \N}$ is defined in terms of $(X_n)_{n \in \N}$ in the same way as in Proposition \ref{prop:convlaw}, then one has the convergence in law $Y_n \rightharpoonup \phi_{c,\alpha,\beta}$ \cite[Proposition 2.3]{FMN17}. On the other hand, the mod-$\phi_{c,\alpha,\beta}$ convergence implies the existence of a zone of control $[-K,K]$ with $\gamma = 0$, with index $(\nu=0,\omega=\alpha)$ and with $K$ as large as wanted (and $K_2=0$). Therefore, Definition \ref{def:zone} is a generalisation of the notion of mod-stable convergence. Conversely, a zone of control does not imply the mod-stable convergence, even if $\gamma\geq 0$. However, in all the examples that we are going to present, it will always be the case that the sequence under consideration converges mod-$\phi_{c,\alpha,\beta}$ with the same parameters $(t_n)_{n \in \N}$ as for the notion of zone of control. 
\medskip

\subsection{Local limit theorems for mod-stable random variables}\label{subsec:llt_modstable}
We can now state our main result:

\begin{theorem}\label{thm:locallimit1}
Let  $(X_n)_{n \in \N}$ be a sequence of real-valued random variables, $\phi_{\alpha,\beta,c}$ a stable reference law, $(t_n)_{n \in \N}$ a sequence growing to infinity, and 
$$\theta_n(\xi) = \esper[\E^{\I \xi X_n}]\,\E^{-t_n\eta_{c,\alpha,\beta}(\I \xi)}.$$ We assume that there is a zone of control $\left[-K\left(t_n\right)^\gamma,K\left(t_n\right)^\gamma\right]$ with index $(\nu,\omega)$, and we denote $(Y_n)_{n \in \N}$ the renormalisation of $(X_n)_{n \in \N}$ given by Proposition \ref{prop:convlaw}. Let $x \in \R$ and $B$ be a fixed Jordan measurable subset with $m(B)>0$. Then, for every exponent $\delta \in (0,\gamma+\frac{1}{\alpha})$, 
$$ \lim_{n \to \infty} (t_n)^{\delta}\,\,\proba\!\left[Y_n - x \in \frac{1}{(t_n)^\delta}\,B\right] = p_{c,\alpha,\beta}(x)\,m(B).$$
\end{theorem}

Before proving Theorem \ref{thm:locallimit1}, let us make a few comments. First, since this is an asymptotic result, we actually only need a zone of control on the residues $\theta_n$ for $n$ large enough. Secondly, for exponents $\delta \in (0,\frac{1}{\alpha}]$, this local limit theorem was proven in \cite[Theorem 5, Propositions 1 and 2]{DKN15}. Theorem \ref{thm:locallimit1} improves on these previous results by showing that the stable approximation holds at scales $(t_n)^{-\delta}$:
\begin{itemize}
     \item which can be smaller than in \cite{DKN15},
     \item and which are directly connected to the size of the zone of control.
\end{itemize}  
 
\begin{lemma}\label{lem:estimatetestfunction}
Consider a sequence $(X_n)_{n\in \N}$ that satisfies the assumptions of Theorem \ref{thm:locallimit1}. Let $f_n \in \testf_0(\R)$ be a smooth test function whose Fourier transform $\widehat{f}_n$ has its support included in the zone $\left[-K\left(t_n\right)^{\gamma+1/\alpha},K\left(t_n\right)^{\gamma+1/\alpha}\right]$. There exists a constant $C(c,\alpha,\nu)$ such that 
$$\left|\esper[f_n(Y_n)] - \int_{\R} f_n(y)\,\phi_{c,\alpha,\beta}(\!\DD{y}) \right| \leq C(c,\alpha,\nu)\,K_1\,\frac{\|f_n\|_{\leb^1}}{(t_n)^{\nu/\alpha}}.$$
\end{lemma}
\begin{proof}
See \cite[Proposition 2.12]{FMN17}.
\end{proof}

\begin{proof}[Proof of Theorem \ref{thm:locallimit1}]
We fix $x,\delta$ and $B$ as in the statement of the theorem. Suppose that we can prove that
$$\lim_{n \to \infty} (t_n)^{\delta}\,\esper\!\left[f((t_n)^\delta (Y_n-x))\right] = p_{c,\alpha,\beta}(x)\,\left(\int_{\R}f(y)\DD{y}\right).$$
for any $f \in \testf_0(\R)$. Then, for any $\eta>0$, Corollary \ref{cor:approx} shows that there exist $f_1,f_2 \in \testf_0(\R)$ with $f_1\leq 1_{B} \leq f_2$ and $\int_{\R}f_2(x)-f_1(x)\DD{x} \leq \eta$, so 
\begin{align*}
\limsup_{n \to \infty} \,(t_n)^{\delta}\,\,&\proba\!\left[Y_n-x \in \frac{1}{(t_n)^\delta}\,B\right] 
\leq \lim_{n \to \infty} \,(t_n)^{\delta}\,\,\esper\!\left[f_2((t_n)^\delta (Y_n-x))\right] \\
&\leq p_{c,\alpha,\beta}(x)\,\left(\int_{\R} f_2(y)\DD{y}\right) \\
&\leq p_{c,\alpha,\beta}(x)\,\left(\int_{\R} f_1(y)\DD{y}+\eta\right)  \\
&\leq p_{c,\alpha,\beta}(x)\,\eta + \lim_{n \to \infty} (t_n)^{\delta}\,\,\esper\!\left[f_1((t_n)^\delta (Y_n-x))\right] \\
&\leq p_{c,\alpha,\beta}(x)\,\eta + \liminf_{n \to \infty}\, (t_n)^{\delta}\,\,\proba\!\left[Y_n-x \in \frac{1}{(t_n)^\delta}\,B\right] 
\end{align*}
so the local limit theorem holds. Hence, as in the proof of the Stone--Feller local limit theorem, we have reduced our problem to estimates on test functions in $\testf_0(\R)$. Fix $f \in \testf_0(\R)$, and denote $f_n(y)=f((t_n)^{\delta}(y-x))$. If $[-C,C]$ is the support of $\widehat{f}$, then $[-C(t_n)^{\delta},C(t_n)^{\delta}]$ is the support of $\widehat{f}_n$, and since $\delta<\gamma+\frac{1}{\alpha}$, it is included in $[-K(t_n)^{\gamma+1/\alpha},K(t_n)^{\gamma+1/\alpha}]$ for $n$ large enough ($K$ being given by Condition \ref{hyp:zone1} of zone of control). Hence, by the previous lemma, 
\begin{align*}
\esper[f_n(Y_n)] &= \left(\int_{\R} f_n(y)\,\phi_{c,\alpha,\beta}(\!\DD{y})\right) + O\left(\frac{\|f_n\|_{\leb^1}}{(t_n)^{\frac{\nu}{\alpha}}}\right) =\left(\int_{\R} f_n(y)\,\phi_{c,\alpha,\beta}(\!\DD{y})\right) + O\left(\frac{\|f\|_{\leb^1}}{(t_n)^{\frac{\nu}{\alpha}+\delta}}\right),
\end{align*}
which implies
\begin{align*}
(t_n)^{\delta} \,\,\esper[f((t_n)^\delta (Y_n-x))]&=(t_n)^{\delta}\,\,\esper[f_n(Y_n)] \\
&= (t_n)^{\delta}\, \left(\int_{\R} f_n(y)\,\phi_{c,\alpha,\beta}(\!\DD{y})\right) + O\!\left(\frac{\|f\|_{\leb^1}}{(t_n)^{\frac{\nu}{\alpha}}}\right)\\
&=\int_{\R} f(u)\,p_{c,\alpha,\beta}\left(x+\frac{u}{(t_n)^{\delta}}\right)\DD{u} + o(1).
\end{align*}
Since $p_{c,\alpha,\beta}(y)=\frac{1}{2\pi}\int_{\R}\E^{\eta_{c,\alpha,\beta}(\I\xi)}\,\E^{-\I y \xi}\DD{\xi}$ is bounded by $\frac{1}{2\pi}\int_{\R} \E^{-|c\xi|^\alpha}\DD{\xi}$, by dominated convergence, the limit of the integral is 
\begin{equation*}
p_{c,\alpha,\beta}(x)\,\left(\int_{\R} f(u)\DD{u}\right).\qedhere
\end{equation*}
\end{proof}
\medskip

If we want Theorem \ref{thm:locallimit1} to be meaningful, it is natural to ask when one can give an explicit formula for the density $p_{c,\alpha,\beta}$ at the real point $x$. Unfortunately, there is no general closed formula for the density of a stable law and it is known explicitly only for the Lévy, Cauchy and normal distributions. However, the following proposition gives a sufficient condition for the existence of a closed formula at the origin.

\begin{proposition} Suppose that $|\beta \tan(\frac{\alpha\pi}{2})|< 1$. Then, the density of the stable distribution $\phi_{c,\alpha,\beta}$ at $x=0$ is given by the convergent series
$$ p_{c,\alpha,\beta}(0)= \frac{1}{\pi \alpha c}\,\sum_{k=0}^\infty (-1)^k\,\left(\beta \,\tan\left(\frac{\pi \alpha}{2}\right)\right)^{2k}\,\frac{\Gamma(2k+\frac{1}{\alpha})}{\Gamma(2k+1)}.$$ 
\end{proposition}
\begin{proof}
Suppose first that $\alpha \neq 1$. One computes
\begin{align*}
p_{c,\alpha,\beta}(0)&=\frac{1}{2\pi} \int_{\R}\E^{\eta(\I \xi)}\DD{\xi} = \frac{1}{\pi} \int_{0}^\infty \E^{-(c\xi)^\alpha}\,\cos\left((c\xi)^\alpha \beta\,\tan\left(\frac{\alpha\pi}{2}\right)\right) \DD{\xi} \\
&=\frac{1}{\pi\alpha c} \int_{0}^\infty \E^{-u}\,\cos\left(u \beta\,\tan\left(\frac{\alpha\pi}{2}\right)\right) u^{\frac{1}{\alpha}-1}\DD{u}.
\end{align*}
Under the assumption $|\beta \tan(\frac{\alpha\pi}{2})|<1$, one can develop in power series the cosinus and change the order of summation to obtain the formula claimed; this ends the proof when $\alpha \neq 1$. If $\alpha=1$, then $|\beta \tan(\frac{\alpha\pi}{2})|< 1$ is satified if and only if $\beta=0$. In this case, one deals with the Cauchy law 
$$\frac{1}{\pi c}\,\frac{1}{1+\frac{x^2}{c^2}}\DD{x},$$ 
which has density $\frac{1}{c\pi}$ at $x=0$. This is also what is obtained by specialisation of the power series, because, if $\beta=0$, then for every $\alpha \in (0,2]$, the power series specialises to
\begin{equation*}
\frac{1}{\pi\alpha c}\,\,\Gamma\!\left(\frac{1}{\alpha}\right).\qedhere
\end{equation*}
\end{proof}
\bigskip

\section{Sums of random variables}\label{sec:sum}
In this section, we apply our main result to various examples of random variables which admit a representation in law as a sum of elementary components which are independent (Sections \ref{subsec:zeroes}, \ref{subsec:partition} and \ref{subsec:randomzeta}) or dependent (Sections \ref{subsec:dependencygraph} and \ref{subsec:markov}).

\subsection{Size of a random integer partition or plane partition}\label{subsec:partition}
To illustrate our theory of zone of controls and the related local limit theorems, we shall consider as a first example the size of a random integer partition or plane partition chosen with probability proportional to $q^{\mathrm{vol}(\lambda)}$. Let us start with \emph{integer partitions}; we refer to \cite[\S1.1]{Mac95} for the details of their combinatorics. An integer partition of size $n$ is a non-increasing sequence $\lambda = (\lambda_1\geq \lambda_2 \geq \cdots \geq \lambda_r)$ of positive integers such that $\lambda_1+\lambda_2+\cdots+\lambda_r = n$. We then denote $n=|\lambda|$, and we represent $\lambda$ by its Young diagram, which is the array of boxes with $\lambda_1$ boxes on the first row, $\lambda_2$ boxes on the second row, \emph{etc.} For instance, $\lambda = (5,5,3,2)$ is an integer partition of size $15$ represented by the Young diagram
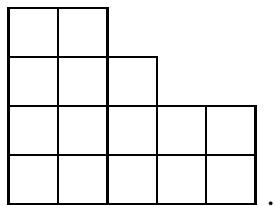
\ytableausetup{aligntableaux=bottom}
$$
\ydiagram{2,3,5,5}\,\,.
$$
Let $\mathfrak{Y}$ be the set of all integer partitions, and $\proba_q$ be the probability measure on $\mathfrak{Y}$ which is proportional to $q^{|\lambda|}$, $q$ being a fixed parameter in $(0,1)$. The corresponding partition function is given by Euler's formula
$$Z(q) = \sum_{\lambda \in \mathfrak{Y}} q^{|\lambda|} = \prod_{n=1}^\infty \frac{1}{1-q^n}.$$
Thus, $\proba_q[\lambda] = \left(\prod_{n=1}^\infty 1-q^n\right)\,q^{|\lambda|}$. We are interested in the asymptotics of the size $S_q$ of a random integer partition taken according to the probability measure $\proba_q$. The Laplace transform of $S_q$ is
$$\esper[\E^{z S_q}] = \prod_{n=1}^\infty \frac{1-q^n}{1-q^n\E^{nz}};$$
it is well defined for $\mathrm{Re}(z)<-\log q$. This formula shows that $S_q$ has the same law as a random series
$$S_q = \sum_{n=1}^\infty nY_n,$$
where the $Y_n$'s are independent, and $Y_n$ is a geometric random variable of parameter $1-q^n$, with distribution
$\proba[Y_n=k] = (1-q^n)q^{nk}$ for any $k \in \N$.
Set $A_n=nY_n$, and $f_n(\xi) = \log \esper[\E^{\I \xi A_n}] - \frac{nq^n\,\I \xi}{1-q^n} + \frac{n^2q^n}{(1-q^n)^2}\,\frac{\xi^2}{2} $. The function $f_n$ and its two first derivatives vanish at $0$, and
$$f_n'''(\xi) = -\I n^3 \frac{(q\E^{\I \xi})^n + (q\E^{\I \xi})^{2n}}{(1-(q\E^{\I \xi})^n)^3}\qquad;\qquad |f_n'''(\xi)| \leq n^3 \frac{q^n + q^{2n}}{(1-q^n)^3}.$$
Set $M_q = \sum_{n=1}^\infty \frac{nq^n}{1-q^n}$ and $V_q = \sum_{n=1}^\infty \frac{n^2q^n}{(1-q^n)^2}$. 
\begin{lemma}\label{lem:meanvariancepartition}
The mean $M_q$ and the variance $V_q$ of the random variable $S_q$ have for asymptotic behavior
\begin{align*}
M_q &= \frac{\zeta(2)}{(\log q)^2} + \frac{1}{2\,\log q} +\frac{1}{24} + o(1) =\frac{\zeta(2)}{(1-q)^2}\,(1+o(1));\\
V_q &=\frac{2\zeta(2)}{(1- q)^3}\,(1+o(1))
\end{align*}
as $q$ goes to $1$.
\end{lemma}
\begin{proof}
For the first asymptotic expansion, we follow closely \cite[Theorem 2.2]{BW17}. Introduce the series $L(q^x) = \sum_{k=1}^\infty q^{kx}=\frac{q^x}{1-q^x}$, and consider the operator $D = \frac{\partial}{\partial x}$, and 
$$\frac{D}{\E^D-1} = \sum_{n=0}^\infty \frac{B_n\,D^n}{n!},$$ 
where the $B_n$'s are the Bernoulli numbers. We have
\begin{align*}
\left(\frac{D}{\E^D-1}\right) (L(q^x)) &= \sum_{k=1}^\infty \left(\frac{D}{\E^D-1}\right)(q^{kx}) = \sum_{k=1}^\infty \sum_{n=0}^\infty \frac{B_n}{n!}\,D^n(q^{kx}) \\
&= \sum_{k=1}^\infty \left(\sum_{n=0}^\infty \frac{B_n (k \log q)^n}{n!}\right)\,q^{kx} = (-\log q) \sum_{k=1}^\infty \frac{k\,q^{kx}}{1-q^k}.
\end{align*}
On the other hand, we have the expansion in powers of $x \log q$:
$$L(q^x) = -\frac{1}{x\log q} + \sum_{k=0}^\infty \frac{\zeta(-k)}{k!}\,(x\log q)^k.$$
We have the relation $\frac{D}{\E^D-1}(x^k) = -k\,\zeta(1-k,x)$ where $\zeta(s,x)=\sum_{n=0}^\infty \frac{1}{(n+x)^s}$ is Hurwitz' zeta function (extended to a meromorphic function of the complex parameter $s$). Therefore, we obtain:
\begin{align*}
\sum_{k=1}^\infty \frac{k\,q^{kx}}{1-q^k} &= \frac{\zeta(2,x)}{(\log q)^2} + \frac{1}{2\log q}+\sum_{k=0}^\infty \frac{\zeta(-k-1)\,\zeta(-k,x)}{k!}\,(\log q)^{k},
\end{align*}
hence the first equivalent by taking $x=1$. 
For the equivalent of the variance, let us remark that
\begin{align}
V_q &= \sum_{n=1}^\infty \frac{n^2 q^n}{(1-q^n)^2} = \sum_{n=1}^\infty \sum_{k=1}^\infty n^2 k\, q^{nk} = \sum_{k=1}^\infty k\,\frac{q^k+q^{2k}}{(1-q^k)^3} \nonumber \\
&= \frac{1}{(1-q)^3}\sum_{k=1}^\infty \frac{k}{(1+q+\cdots+q^{k-1})^3}\,(q^{k}+q^{2k}).\label{eq:variancepartition}
\end{align}
As $q$ goes to $1$, each term of the series in \eqref{eq:variancepartition} converges to $\frac{2}{k^2}$, and it is an easy exercice to see that one can sum these limits. 
\end{proof}
\medskip

Set $X_q = \frac{S_q-M_q}{(V_q)^{4/9}}$. We have:
\begin{align}
\log \esper[\E^{\I \xi X_q}] + \frac{(V_q)^{1/9}\xi^2}{2} &= \sum_{n=1}^\infty f_n\left(\frac{\xi}{(V_q)^{4/9}}\right) \nonumber \\
\left|\log \esper[\E^{\I \xi X_q}] + \frac{(V_q)^{1/9}\xi^2}{2}\right| &\leq \frac{|\xi|^3}{6\,(V_q)^{4/3}} \sum_{n=1}^\infty \frac{n^3(q^n+q^{2n})}{(1-q^n)^3}.\label{eq:thirdmomentpartition}
\end{align}
As $q$ goes to $1$, the series in Equation \eqref{eq:thirdmomentpartition} behaves as 
\begin{align*}
\sum_{n=1}^\infty &\frac{n^3(q^n+q^{2n})}{(1-q^n)^3} = \sum_{n=1}^\infty \sum_{k=1}^\infty n^3\,\frac{k(k+1)}{2} (q^{kn} + q^{(k+1)n}) \\
&= \sum_{k=1}^\infty \frac{k(k+1)}{2} \left(\frac{q^k(1+4q^k+q^{2k})}{(1-q^k)^4}+\frac{q^{k+1}(1+4q^{k+1}+q^{2k+2})}{(1-q^{k+1})^4}\right)\\
&=\frac{3(1+o(1))}{(1-q)^4} \sum_{k=1}^\infty \left(\frac{1}{k^2}+\frac{1}{k^3} + \frac{1}{(k+1)^2} - \frac{1}{(k+1)^3}\right) = \frac{6\,\zeta(2)\,(1+o(1))}{(1-q)^4}.
\end{align*}
It follows that for any constant $C>\frac{1}{2^{4/3}(\zeta(2))^{1/3}}$, there exists $q_0 \in (0,1)$ such that if $q \geq q_0$, then 
\begin{align*}
\left|\log \esper[\E^{\I \xi X_q}] + \frac{(V_q)^{1/9}\xi^2}{2}\right| &\leq C|\xi^3|; \\
|\theta_q(\xi)-1| &\leq C|\xi|^3\,\exp(C|\xi|^3),
\end{align*}
uniformly on the parameter $\xi \in \R$, with $\theta_q(\xi)=\esper[\E^{\I\xi X_q}]\,\E^{\frac{(V_q)^{1/9}\xi^2}{2}}$. Hence, the family $(X_q)_{q \in (0,1)}$ has a zone of control of mod-Gaussian convergence for the parameter $t_q = (V_q)^{1/9}$, with index $(3,3)$ and with size $O((t_q)^{3/2})=O((V_q)^{1/6})$ if one wants Condition \ref{hyp:zone2} to be satisfied. We conclude with Theorem \ref{thm:locallimit1} and \cite[Theorem 2.16]{FMN17}:
\begin{proposition}
Let $S_q$ be the size of a random integer partition chosen with probability proportional to $q^{|\lambda|}$, and $M_q$ and $V_q$ be defined as in Lemma \ref{lem:meanvariancepartition}. As $q$ goes to $1$, the random variable $Y_q = (S_q-M_q)/\sqrt{V_q}$ converges in law to the standard Gaussian distribution, and one has more precisely:
$$\dkol(Y_q\,,\,\gauss) = O\!\left((1-q)^{1/2}\right).$$
Moreover, one has the following local limit theorem: for any exponent $\delta \in (0,\frac{1}{2})$ and any Jordan measurable subset $B$ with $m(B)>0$,
$$\lim_{q \to 1} \,(1-q)^{-\delta}\,\,\proba\!\left[Y_q - x \in (1-q)^{\delta}\,B\right] = \frac{\E^{-\frac{x^2}{2}}}{\sqrt{2\pi}}\,m(B).$$
\end{proposition}
\noindent Note that this result does not allow one to go up to the discrete scale. Indeed, the estimate of the variance shows that the Gaussian approximation for $S_q-M_q$ holds at scales $(1-q)^{\delta-3/2}$ with $\delta < \frac{1}{2}$, so one cannot describe what happens for the scales $(1-q)^{-\gamma}$ with $\gamma \in (0,1]$.

\begin{remark}
The asymptotics of the expectations $M_q$ are easy to retrieve from the Hardy--Ramanujan asymptotic formula for the number $p(n)$ of integer partitions of size $n$:
$$p(n) = \frac{(1+o(1))}{4n\sqrt{3}}\,\exp\!\left(\pi \sqrt{\frac{2n}{3}}\right),$$
see \cite{HR18,Rad38}. It implies that the probability measure $\proba_q[\lambda] = \frac{q^{|\lambda|}}{Z(\lambda)}$ is concentrated on partitions of size $n$ such that $p(n)\,q^n$ is maximal, that is roughly with
$$f_q(n)=\pi \sqrt{\frac{2n}{3}} + n \log q $$
maximal. When $q$ is fixed, the maximal of $f_q(n)$ is attained at $n=\frac{\pi^2}{6(\log q)^2}=\frac{\zeta(2)}{(\log q)^2}$; this is the leading term in the asymptotic expansion of $M_q$.
\end{remark}
\medskip

Similarly, one can study the size of a random \emph{plane partition} chosen with a probability proportional to $q^{\mathrm{vol}(\lambda)}$. A plane partition is a sequence $\lambda = (\lambda^{(1)},\lambda^{(2)},\ldots,\lambda^{(s)})$ of non-empty integer partitions such that the following inequalities hold: 
$$\forall i \leq s-1,\,\,\forall j \leq \ell(\lambda^{(i)}),\,\,\,\lambda^{(i)}_j \geq \lambda^{(i+1)}_j .$$
We refer to \cite[Chapter 11]{And76} for the combinatorics of these objects. They can be represented by $3$-dimensional Young diagrams, so for instance,
\vspace{2mm}
\newcounter{x}
\newcounter{y}
\newcounter{z}
\newcommand\xaxis{210}
\newcommand\yaxis{-30}
\newcommand\zaxis{90}
\newcommand\topside[3]{
  \fill[fill=NavyBlue, draw=black,shift={(\xaxis:#1)},shift={(\yaxis:#2)},
  shift={(\zaxis:#3)}] (0,0) -- (30:1) -- (0,1) --(150:1)--(0,0);
}
\newcommand\leftside[3]{
  \fill[fill=red!50!white, draw=black,shift={(\xaxis:#1)},shift={(\yaxis:#2)},
  shift={(\zaxis:#3)}] (0,0) -- (0,-1) -- (210:1) --(150:1)--(0,0);
}
\newcommand\rightside[3]{
  \fill[fill=NavyBlue!50!white, draw=black,shift={(\xaxis:#1)},shift={(\yaxis:#2)},
  shift={(\zaxis:#3)}] (0,0) -- (30:1) -- (-30:1) --(0,-1)--(0,0);
}
\newcommand\cube[3]{
  \topside{#1}{#2}{#3} \leftside{#1}{#2}{#3} \rightside{#1}{#2}{#3}
}
\newcommand\planepartition[1]{
 \setcounter{x}{-1}
  \foreach \a in {#1} {
    \addtocounter{x}{1}
    \setcounter{y}{-1}
    \foreach \b in \a {
      \addtocounter{y}{1}
      \setcounter{z}{-1}
      \foreach \c in {1,...,\b} {
        \addtocounter{z}{1}
        \cube{\value{x}}{\value{y}}{\value{z}}
      }
    }
  }
}

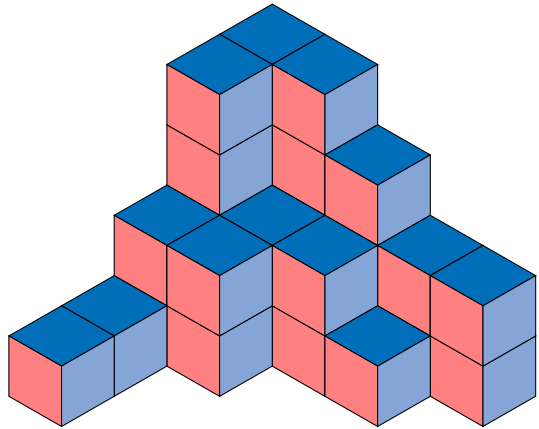
\begin{center}
\begin{tikzpicture}[scale=0.8]
\planepartition{{4,4,3,2,2},{4,2,2,1},{2,2},{1},{1}}
\end{tikzpicture}\vspace{2mm}
\end{center}
is the diagram of the plane partition $((5,5,3,2),(4,3,1,1),(2,2),(1),(1))$. The volume of a plane partition is the number of boxes of its diagram, that is $\mathrm{vol}(\lambda)=|\lambda^{(1)}| + \cdots + |\lambda^{(r)}|$. The generating series of the volumes of the plane partitions is given by MacMahon's formula:
$$\sum_{\lambda \text{ plane partition}}q^{\mathrm{vol}(\lambda)} = \prod_{n=1}^\infty \frac{1}{(1-q^n)^n}.$$
Therefore, if $S_q'$ is the size of a random plane partition chosen according to the probability measure $\proba_q'[\lambda] = \frac{q^{\mathrm{vol}(\lambda)}}{Z'(q)} = (\prod_{n=1}^\infty (1-q^n)^n)\,q^{\mathrm{vol}(\lambda)}$, then the Laplace transform of $S_q'$ is
$$\esper[\E^{z S_q'}] = \prod_{n=1}^\infty \left(\frac{1-q^n}{1-q^n\E^{nz}}\right)^n.$$
Thus, $S_q'$ admits a representation in law as a random series of independent random variables
$$S_q' = \sum_{n=1}^\infty \sum_{i=1}^n n\,Y_{n,i},$$
where $Y_{n,i}$ is a geometric random variable with parameter $(1-q^n)$. Set $M_q' = \sum_{n=1}^\infty \frac{n^2\,q^n}{1-q^n}$ and $V_q' = \sum_{n=1}^\infty \frac{n^3\,q^n}{(1-q^n)^2}$. Similar arguments as those used in the proof of Lemma \ref{lem:meanvariancepartition} show that
\begin{align*}
M_q' &= \frac{2\,\zeta(3)}{(-\log q)^3} + \frac{1}{12\,\log q} +o(1) = \frac{2\,\zeta(3)}{(1- q)^3}\,(1+o(1));\\
V_q' &= \frac{6\,\zeta(3)}{(1- q)^4}\,(1+o(1))
\end{align*}
as $q$ goes to $1$. Again, the asymptotics of $M_q'$ are related to the asymptotic formula for the number on plane partitions with volume $n$:
$$p'(n) = \frac{(1+o(1))\,(\zeta(3))^{7/36}}{\sqrt{12\pi}}\,\left(\frac{n}{2}\right)^{-\frac{25}{36}}\,\exp\!\left(3(\zeta(3))^{1/3}\left(\frac{n}{2}\right)^{2/3}+\zeta'(-1)\right),$$
see \cite{Wright31,KM06}. Set $$X_q' = \frac{S_q'-M_q'}{(V_q')^{5/12}}.$$ Since $S_q'$ involves the same geometric random variables as before, we can perform the same computations as before to prove that $(X_q')_{q \in (0,1)}$ admits a zone of control of mod-Gaussian convergence for the parameter $t_q = (V_q')^{1/12}$. This zone of control has again index $(3,3)$, and its size can be taken equal to $O((V_q')^{1/8})$. We conclude:
\begin{proposition}
Let $S_q'$ be the size of a random plane partition chosen with probability proportional to $q^{\mathrm{vol}(\lambda)}$, and $M_q'$ and $V_q'$ be  the expectation and the variance of $S_q'$. As $q$ goes to $1$, $Y_q' = \frac{S_q'-M_q'}{\sqrt{V_q'}}$ converges in law to the standard Gaussian distribution, and one has more precisely:
$$\dkol(Y_q'\,,\,\gauss) = O\!\left((1-q)^{1/2}\right).$$
Moreover, for any exponent $\delta \in (0,\frac{1}{2})$ and any Jordan measurable subset $B$ with $m(B)>0$,
$$\lim_{q \to 1} \,(1-q)^{-\delta}\,\,\proba\!\left[Y_q' - x \in(1-q)^{\delta}\,B\right] = \frac{\E^{-\frac{x^2}{2}}}{\sqrt{2\pi}}\,m(B).$$
\end{proposition}
\medskip

\subsection{Determinantal point processes and zeroes of a random analytic series}\label{subsec:zeroes} The determinantal point processes form another framework which often yields mod-Gaussian random variables satisfying Theorem \ref{thm:locallimit1}. Consider a locally compact, separable and complete metric space $\mathfrak{X}$ endowed with a locally finite measure $\lambda$, and a Hermitian non-negative linear operator $\mathscr{K} : \leb^2(\mathfrak{X},\lambda) \to \leb^2(\mathfrak{X},\lambda)$ such that for any relatively compact subset $A \subset \mathfrak{X}$, the induced operator $\mathscr{K}_A = 1_A\,\mathscr{K}\,1_A$ on $\leb^2(A,\lambda_{|A})$ is a trace class operator with spectrum included in $[0,1]$. The operator $\mathscr{K}$ can then be represented by a Hermitian locally square-integrable kernel $K$:
$$(\mathscr{K}f)(x) = \int_\mathfrak{X} K(x,y)\,f(y)\,\lambda(\!\DD{y}).$$
In this setting, there is a unique law of random point process $M=\sum_{i \in I} \delta_{X_i}$ on $\mathfrak{X}$ such that the correlation functions of $M$ with respect to the reference measure $\lambda$ write as
$$\rho_n(x_1,\ldots,x_n) = \det(K(x_i,x_j))_{1\leq i,j\leq n}.$$
One says that $M$ is the determinantal point process associated to the kernel $K$, see for instance \cite{Sosh00,Joh05} and \cite[Chapter 4]{HKPV09} for details. For any relatively compact set $A$, the number of points $M(A)$ of the random point process that falls in $A$ writes then as
$$M(A) =_{(\mathrm{law})} \sum_{j\in J} \mathrm{Ber}(p_{A,j}),$$
where the $p_{A,j}$'s are the eigenvalues of the trace class integral operator $\mathscr{K}_A$, and the Bernoulli random variables are independent; see \cite[Theorem 4.5.3]{HKPV09}. 

\begin{proposition}\label{prop:determinantal}
 We consider a determinantal point process $M$ as above, with a continuous kernel $K$ that is locally square-integrable but not square-integrable: $\int_{\spa^2} |K(x,y)|^2\,\lambda(\!\DD{x})\,\lambda(\!\DD{y})=+\infty$. We also fix a growing sequence $(A_n)_{n \in \N}$ of relatively compact subsets of $\spa$ such that $\bigsqcup_{n \in \N} A_n = \spa$, and such that the ratio 
 $$r_n = \left(\frac{\int_{A_n} K(x,x)\,\lambda(\!\DD{x})}{\int_{(A_n)^2} |K(x,y)|^2\,\lambda(\!\DD{x})\,\lambda(\!\DD{y})}\right) \to_{n \to +\infty} r $$
 admits a limit $r \in (1,+\infty]$ (we allow $r=+\infty$, and we shall see that $r_n \geq 1$ for any $n \in \N$). Then, with $m_n = \esper[M(A_n)]$ and $v_n = \Var(M(A_n))$, we have mod-Gaussian convergence of $X_n = (M(A_n)-m_n)/(v_n)^{1/3}$ with parameters $t_n=(v_n)^{1/3}$, and with a zone of control of size $O((v_n)^{1/3})$ and with index $(3,3)$. Therefore, for any $\delta \in (0,\frac{1}{2})$ and any Jordan measurable subset $B$ with $m(B)>0$,
 $$\lim_{n \to \infty} (v_n)^{\frac{1}{2}-\delta}\,\proba\!\left[M(A_n) - m_n - x(v_n)^{\frac{1}{2}} \in (v_n)^{\delta}\,B\right] = \frac{\E^{-\frac{x^2}{2}}}{\sqrt{2\pi}}\,m(B).$$
 If $Y_n = \frac{M(A_n)-m_n}{\sqrt{v_n}}$, then we also have $\dkol(Y_n,\gauss)=O((v_n)^{-1/2})$.
 \end{proposition}

\begin{proof}
Denote $(p_{n,j})_{j \in \N}$ the non-increasing sequence of eigenvalues of the compact operator $\mathscr{K}_{A_n}$, these eigenvalues being counted with multiplicity; they all belong to $[0,1]$ by \cite[Theorem 3]{Sosh00}. The expectation of $M(A_n)$ is
$$m_n= \sum_{j \in \N} p_{n,j} =\int_{A_n} \rho_1(x)\,\lambda(\!\DD{x}) = \int_{A_n} K(x,x)\,\lambda(\!\DD{x}) = \sum_{j \in \N} p_{n,j},$$
and its variance is
\begin{align*}
v_n&= \sum_{j \in \N} p_{n,j}(1-p_{n,j}) =  \esper[M(A_n)] + \esper[M(A_n)(M(A_n)-1)] - (\esper[M(A_n)])^2 \\
&= \int_{A_n} K(x,x)\,\lambda(\!\DD{x}) + \int_{(A_n)^2} \left(\det\left(\begin{smallmatrix}K(x,x) & K(x,y) \\ K(y,x) & K(y,y)\end{smallmatrix} \right) - K(x,x)\,K(y,y) \right)\,\lambda(\!\DD{x})\,\lambda(\!\DD{y}) \\
&= \int_{A_n} K(x,x)\,\lambda(\!\DD{x}) - \int_{(A_n)^2} |K(x,y)|^2\,\lambda(\!\DD{x})\,\lambda(\!\DD{y}).
\end{align*}
In particular, we always have $r_n = \frac{m_n}{m_n - v_n} \geq 1$. Since $K$ is not square-integrable on $\spa^2$, 
$$\lim_{n \to \infty} m_n - v_n = \lim_{n \to \infty} \int_{(A_n)^2} |K(x,y)|^2\,\lambda(\!\DD{x})\,\lambda(\!\DD{y}) = +\infty.$$ 
If $r \in (1,+\infty)$, then $m_n$ and $v_n$ grow to infinity at the same speed $s_n=m_n-v_n$, but with different rates $rs_n$ and $(r-1)s_n$. If $r=+\infty$, then $m_n$ and $v_n$ grow to infinity faster than the speed $s_n$, and $m_n/v_n \to 1$.\bigskip

Consider a real parameter $\zeta$ with $|\zeta| \leq c$ for some constant $c<1$ sufficiently small, so that $\sum_{n=3}^\infty \frac{c^{n-3}}{n} \leq \frac{1}{2}$. Then, with $p \in [0,1]$, the power series expansion of $\log(1+t)$ yields
\begin{align*}
&\left|\log(1+p(\E^{\I \zeta}-1)) - \I p\zeta +p(1-p)\frac{\zeta^2}{2}\right| \\
&\leq p \left|\E^{\I \zeta}-1-\I\zeta+\frac{\zeta^2}{2}\right| + \frac{p^2}{2} \left|(\E^{\I \zeta}-1)^2 + \zeta^2\right| + \frac{p^3 |\zeta|^3}{2} \leq A\, p\,|\zeta|^3
\end{align*}
for some positive constant $A$, since $|\E^{\I \zeta}-1| \leq |\zeta|$ for any $\zeta$. Therefore, with $\zeta = \frac{\xi}{(v_n)^{1/3}}$, we obtain on the zone $\xi \in [-c(v_n)^{1/3},c(v_n)^{1/3}|$ the estimate
$$\left|\sum_{j \in \N} \log\!\left(1+p_{n,j}\left(\E^{\frac{\I \xi}{(v_n)^{1/3}}}-1\right)\right) - \I \,\frac{m_n}{(v_n)^{1/3}}\, \xi + (v_n)^{1/3}\,\frac{\xi^2}{2}\right| \leq A\,\frac{m_n}{v_n}\,|\xi|^3.$$
So, if $t_n=(v_n)^{1/3}$ and $X_n = (M(A_n)-m_n)/(v_n)^{1/3}$, then the identity $\esper[\E^{\I \zeta M(A_n)}]=\prod_{j \in \N} (1+p_{n,j}\,(\E^{\I \zeta}-1))$ leads to
$$|\theta_n(\xi)-1|\leq A\,\frac{m_n}{v_n}\,|\xi|^3\,\exp\left(A\,\frac{m_n}{v_n}\,|\xi|^3\right) \leq K_1|\xi^3|\,\exp(K_2|\xi|^3)$$
with $K_1=K_2 = 2A\,\frac{r}{r-1}$, for $n$ large enough. Once this zone of control is established, the probabilistic estimates follow readily from Theorem \ref{thm:locallimit1} and \cite{FMN17}.
\end{proof}

As an application of the previous proposition, consider $G(z) = \sum_{n=0}^\infty G_n\,z^n$ a random analytic series, with the $G_n$'s that are independent standard complex Gaussian variables. The radius of convergence of $G$ is almost surely equal to $1$, and the set of zeroes of $G$ is a determinantal point process on $D(0,1)=\{z \in \C\,|\,|z|<1\}$ with kernel
$$K(z_1,z_2) = \frac{1}{\pi(1-z_1\overline{z_2})^2};$$
see \cite[Theorem 1]{PV05}. As a consequence, the number of zeroes $Z_R$ of the random series $G$ that fall in the disk $D(0,R)=\{z \in \C\,|\,|z|<R\}$ with $R<1$ admits for representation in law
$$Z_R = \sum_{k=1}^\infty \mathrm{Ber}(R^{2k}),$$
where the Bernoulli random variables are taken independent (Theorem 2 in \emph{loc.~cit.}). In \cite[Section 7.1]{FMN16} and \cite[Section 3.1]{FMN17}, we used this representation to prove the mod-Gaussian convergence of $Z_R$ as $R$ goes to $1$. Here, we remark that
\begin{align*}
m_R &= \esper[Z_R] = \sum_{k=1}^\infty R^{2k} = \frac{R^2}{1-R^2};\\
v_R &= \Var(Z_R) = \sum_{k=1}^\infty R^{2k}(1-R^{2k}) = \frac{R^2}{1-R^2} - \frac{R^4}{1-R^4} = \frac{R^2}{1-R^4},
\end{align*}
so as $R$ goes to $1$, we have $m_R/(m_R-v_R) = 1+R^{-2} \to 2 \in (1,+\infty]$. Consequently, if we introduce the hyperbolic area $h=\frac{4\pi R^2}{(1-R^2)}$ of the disc $D(0,R)$, and if we use the conformal invariance of the point process of zeroes \cite[Section 2.3]{HKPV09}, we obtain:
\begin{proposition}
Denote $Z^h$ the number of zeroes of a random analytic series $G=\sum_{n=0}^\infty G_n\,z^n$ that fall in a disc with hyperbolic area $h$. For any $\delta \in (0,\frac{1}{2})$ and any Jordan measurable subset $B$ with $m(B)>0$,
$$\lim_{h \to +\infty} h^{\frac{1}{2}-\delta}\,\,\proba\!\!\left[Z^h-\frac{h}{4\pi} - \frac{xh^{\frac{1}{2}}}{\sqrt{8\pi}} \in \frac{h^{\delta}}{\sqrt{8\pi}}\,B\right] = \frac{\E^{-\frac{x^2}{2}}}{\sqrt{2\pi}}\,m(B).$$
\end{proposition}
\noindent This result is optimal, because $\delta=0$ corresponds to the discrete scale, where the Gaussian approximation cannot hold.
\medskip

\subsection{Random zeta functions}\label{subsec:randomzeta}
In this section, we consider a multi-dimensional example stemming from number theory and the study of the Riemann $\zeta$ function. Notice that Theorem \ref{thm:approx} and Lemma \ref{lem:estimatetestfunction} hold also in $\R^d$ with $d \geq 2$. Therefore, we have the following extension to $\R^{d \geq 2}$ of our main Theorem \ref{thm:locallimit1} (we only state this extension for mod-Gaussian sequences):
\begin{proposition}\label{prop:locallimitmultidim}
Let $(\XEC_n)_{n \in \N}$ be a sequence of random vectors in $\R^d$, and $(t_n)_{n \in \N}$ a sequence going to $+\infty$. We denote
$$ \theta_n(\XIEC) = \esper[\E^{\I \scal{\XIEC}{\XEC_n} }]\,\E^{\frac{t_n\,\|\XIEC\|^2}{2}},\quad \text{with }\scal{\XIEC}{\XEC_n} = \sum_{i=1}^d \xi_i X_{n,i}\text{ and }\|\XIEC\|^2 = \sum_{i=1}^d (\xi_i)^2. $$
We assume that there is a zone $[-K(t_n)^{\gamma},K(t_n)^\gamma]^d$ such that, for any $\boldsymbol{\xi}$ in this zone,
$$|\theta_n(\XIEC)-1|\leq K_1\,\|\XIEC\|^v\,\exp(K_2\,\|\XIEC\|^w)$$
where $v>0$, $w \geq 2$ and $-\frac{1}{2}<\gamma\leq \frac{1}{w-2}$. \medskip

\noindent Then, $\YEC_n = \XEC_n/\sqrt{t_n}$ converges in law to a standard Gaussian law $\mathcal{N}_{\R^d}(0,I_d)$, and for any $\delta \in (0,\frac{1}{2}+\gamma)$, any $\mathbf{y} \in \R^d$ and any Jordan measurable subset $B \subset \R^d$ with $m(B)>0$,
$$\lim_{n \to \infty} \,\,(t_n)^{d\delta}\,\proba\!\!\left[\YEC_n -\mathbf{y} \in (t_n)^{-\delta}B\right] = \frac{\E^{-\frac{\|\mathbf{y}\|^2}{2}}}{(2\pi)^{\frac{d}{2}}}\,m(B),$$
where $m(B)$ is the $d$-dimensional Lebesgue measure of $B$.
\end{proposition}

We refer to \cite[Theorem 4]{KN12} for a similar statement, with slightly different assumptions. The reader should beware that in the theory of mod-$\phi$ convergent sequences, this local limit theorem is the only multi-dimensional extension of the results in dimension $1$ that is straightforward. Thus, for the speed of convergence estimates and the large deviation results, new phenomenons occur in dimension $d \geq 2$, and the extension of the one-dimensional results is much more involved \cite{FMN17b}. Now, let us apply Proposition \ref{prop:locallimitmultidim} to the sequence of complex random variables
\begin{equation}
    X_n = - \sum_{p \leq n} \log\left(1-\frac{U_p}{\sqrt{p}}\right), \label{eq:randomzeta}
\end{equation}
where the sum runs over prime numbers $p$ smaller than $n$, and the random variables $U_p$ are independent and uniformly distributed on the unit circle. The random variables $X_n$ are simple models of the logarithm of the random zeta function on the critical line, 
see \cite[Section 4.1]{JKN11} and \cite[Example 2]{KN12}. The $2$-dimensional Fourier transform of $X_n$ was computed in \cite{KN12}:
$$\esper[\E^{\I (\xi_1\mathrm{Re}(X_n) + \xi_2 \mathrm{Im}(X_n))}] = \prod_{p \leq n} \hypergeom{\frac{\I \xi_1+\xi_2}{2}}{\frac{\I \xi_1-\xi_2}{2}}{1}\left(\frac{1}{p}\right),$$
where $\hypergeom{a}{b}{c}(z)$ is the hypergeometric function defined by
$\hypergeom{a}{b}{c}(z) = \sum_{m=0}^\infty \frac{a^{\uparrow m}\,b^{\uparrow m}}{c^{\uparrow m}\,m!}\,z^m,$
with $k^{\uparrow m}=k(k+1)\cdots (k+m-1)$. \bigskip

Therefore, if $\theta_n(\XIEC) = \esper[\E^{\I \scal{\XIEC}{X_n}}]\,\E^{\frac{t_n\|\XIEC\|^2}{2}}$ with $t_n = \frac{1}{2}\sum_{p\leq n}\frac{1}{p}$, then
\begin{align*}
\theta_n(\XIEC) = \prod_{p \leq N} \left(1 - \frac{\|\XIEC\|^2}{4p}  +  \sum_{m \geq 2} \frac{(\frac{\I \xi_1+\xi_2}{2})^{\uparrow m}\,(\frac{\I \xi_1-\xi_2}{2})^{\uparrow m}}{(m!)^2}\,p^{-m}\right) \E^{\frac{\|\XIEC\|^2}{4p}}.
\end{align*}
Denote $T_p(\XIEC)$ the terms of the product on the right-hand side. We have
\begin{align*}
|R_p(\XIEC)|&=\left|\sum_{m \geq 2}\frac{(\frac{\I \xi_1+\xi_2}{2})^{\uparrow m}\,(\frac{\I \xi_1-\xi_2}{2})^{\uparrow m}}{(m!)^2}\,p^{-m} \right|\\
 &\leq \sum_{m \geq 2}\left(\frac{(\frac{\|\XIEC\|}{2})^{\uparrow m}}{m!}\right)^{\!2}\,p^{-m} = \frac{1}{2\pi} \int_{0}^{2\pi} \left|\sum_{m\geq 2}\frac{(\frac{\|\XIEC\|}{2})^{\uparrow m}}{m!} \E^{\I m\theta}\,p^{-\frac{m}{2}}\right|^2\DD{\theta}\\
&\leq \frac{1}{2\pi} \int_{0}^{2\pi} \left|\frac{1}{(1-\E^{\I\theta}p^{-1/2})^{\frac{\|\XIEC\|}{2}}} - 1 - \frac{\|\XIEC\|}{2}\,\E^{\I\theta}p^{-1/2}\right|^2\DD{\theta} \\
&\leq \left(\frac{(\frac{\|\XIEC\|}{2})(\frac{\|\XIEC\|}{2}+1)}{2}\,\frac{1}{(1-p^{-1/2})^{\frac{\|\XIEC\|}{2} +2}\,p}\right)^2.
\end{align*}
Suppose $\|\XIEC\|\leq \sqrt{p}$. Then, $\E^{\frac{\|\XIEC\|^2}{4p}} \leq \E^{\frac{1}{4}}$ and $(1-(1-\frac{\|\XIEC\|^2}{4p})\E^{\frac{\|\XIEC\|^2}{4p}})\leq (1-\frac{3}{4}\E^{\frac{1}{4}}) \frac{\|\XIEC\|^4}{p^2}$, so
\begin{align*}
|R_p(\XIEC)| &\leq \frac{1}{64(1-p^{-1/2})^{4+p^{1/2}}}\,\frac{\|\XIEC\|^2(\|\XIEC\|+2)^2}{p^2}\leq \frac{12.06\,\|\XIEC\|^2(\|\XIEC\|+2)^2}{p^2};\\
|T_p(\XIEC)-1|&\leq \left(1-\left(1-\frac{\|\XIEC\|^2}{4p}\right)\E^{\frac{\|\XIEC\|^2}{4p}}\right)+ \E^{1/4}\,|R_p(\XIEC)| \leq \frac{16\,\|\XIEC\|^2(\|\XIEC\|+2)^2}{p^2} \\
&\leq \frac{16\,\|\XIEC\|^2(\|\XIEC\|+2)^2}{p^2}\,\E^{\frac{16\,\|\XIEC\|^2(\|\XIEC\|+2)^2}{p^2}};\\
|T_p(\XIEC)| &\leq 1+\frac{16\,\|\XIEC\|^2(\|\XIEC\|+2)^2}{p^2} \leq \E^{\frac{16\,\|\XIEC\|^2(\|\XIEC\|+2)^2}{p^2}}.
\end{align*}
On the first line, we used the fact that $p \in \mathbb{P} \mapsto (1-p^{-1/2})^{4+p^{1/2}}$ attains its minimum at $p=2$. On the other hand, if $\|\XIEC\| \geq \sqrt{p}$, then 
\begin{align*}
|T_p(\XIEC)| &= \left|\esper\!\left[\E^{-\I \scal{\XIEC}{\log\!\left(1-\frac{U_p}{\sqrt{p}}\right)}}\right]\right|\,\E^{\frac{\|\XIEC\|^2}{4p}} \leq \E^{\frac{\|\XIEC\|^2}{4p}} \leq \E^{\frac{16\,\|\XIEC\|^2(\|\XIEC\|+2)^2}{p^2}};\\
|T_p(\XIEC)-1| &\leq 1 + \E^{\frac{\|\XIEC\|^2}{4p}} \leq 2\,\E^{\frac{\|\XIEC\|^2}{4p}}\leq \frac{16\,\|\XIEC\|^2(\|\XIEC\|+2)^2}{p^2}\,\E^{\frac{16\,\|\XIEC\|^2(\|\XIEC\|+2)^2}{p^2}}.
\end{align*}
From these inequalities, one deduces that \emph{for any} $\XIEC \in \R^2$,
\begin{align*}
|\theta_n(\XIEC)-1| &= \left|\left(\prod_{p \leq n} T_p(\XIEC)\right)-1\right| \leq \sum_{p\leq n} \left(\prod_{\substack{p'\leq n\\ p'\neq p}} |T_{p'}(\XIEC)|\right)\,|T_p(\XIEC)-1| \leq S \exp S
\end{align*}
where $S = \sum_{p\leq n} \frac{16\,\|\XIEC\|^2(\|\XIEC\|+2)^2}{p^2} \leq 8\,\|\XIEC\|^2(\|\XIEC\|+2)^2$. It follows immediately that one has a control of index $(2,4)$ over $\theta_n(\XIEC)-1$, which holds over the whole real line. Hence:
\begin{proposition}
Let $X_n$ be the random log-zeta function defined by Equation \eqref{eq:randomzeta}. For any exponent $\delta \in (0,1)$ and any $z \in \C$,
$$\lim_{n \to \infty} \left(\log \log n\right)^{2\delta}\,\proba\!\left[X_n - z\,\sqrt{\log \log n} \in \left(\log \log n\right)^{\frac{1}{2}-\delta}\,B \right] = \frac{\E^{-|z|^2}}{\pi}\,m(B)$$
for any Jordan measurable bounded set $B \subset \C$ with $m(B)>0$.
\end{proposition}
\noindent This result improves on \cite[Section 3, Example 2]{KN12} and \cite[Section 3.5]{DKN15}, which dealt only with the case $\delta\leq \frac{1}{2}$.
\medskip

\subsection{Sums with sparse dependency graphs}\label{subsec:dependencygraph}
In the previous paragraphs, we looked at sums of independent random variables, but the theory of zones of control is also useful when dealing certain sums of weakly dependent random variables. A general setting where one can prove mod-Gaussian convergence with a zone of control is if one has strong bounds on the \emph{cumulants} of the random variables considered. Recall that if $X$ is a random variable with convergent Laplace transform $\esper[\E^{zX_n}]$, then its cumulants $\kappa^{(r \geq 1)}(X)$ are the coefficients of the log-Laplace transform:
$$\log \esper[\E^{zX_n}] = \sum_{r=1}^\infty \frac{\kappa^{(r)}(X)}{r!}\,z^r.$$
If $(S_n)_{n \in \N}$ is a sequence of random variables, one says that one has uniform bounds on the cumulants with parameters $(D_n,N_n,A)$ if
\begin{equation}
    \forall r \geq 1,\,\,\,|\kappa^{(r)}(S_n)| \leq r^{r-2}\,A^r\,(2D_n)^{r-1}\,N_n.\label{eq:boundoncumulants}
\end{equation}
This definition was introduced in \cite[Definition 4.1]{FMN17}. If $D_n = o(N_n)$ and $N_n \to +\infty$, then the uniform bounds on cumulants imply that $X_n = \frac{S_n-\esper[S_n]}{(N_n)^{1/3}(D_n)^{2/3}}$ admits a zone of control of mod-Gaussian convergence, with parameters
$$ t_n = \frac{\Var(S_n)}{(N_n)^{2/3}(D_n)^{4/3}},$$
index $(3,3)$ and size $O(t_n)$ \cite[Corollary 4.2]{FMN17}. Therefore:
\begin{proposition}
Let $(S_n)_{n \in \N}$ be a sequence of random variables that admit uniform bounds on cumulants (Equation \eqref{eq:boundoncumulants}). We suppose that $\frac{(\mathrm{var}(S_n))^{3/2}}{N_n(D_n)^{2}}$ goes to $+\infty$, and that $D_n = o(N_n)$. Then, if $Y_n = \frac{S_n-\esper[S_n]}{\sqrt{\Var(S_n)}}$, for any $\delta \in (0,1)$ and any Jordan measurable subset $B$ with $m(B)>0$,
$$\lim_{n \to \infty} \left(\frac{(\Var(S_n))^{3/2}}{N_n(D_n)^2}\right)^{\!\delta} \,\,\proba\!\left[Y_n - y \in \left(\frac{(\Var(S_n))^{3/2}}{N_n(D_n)^2}\right)^{\!-\delta}B\right] = \frac{\E^{-\frac{y^2}{2}}}{\sqrt{2\pi}}\, m(B).$$
In particular, if $\liminf_{n \to \infty} \frac{\Var(S_n)}{N_nD_n}>0$, then for any $\gamma \in (0,\frac{1}{2})$,
$$\lim_{n \to \infty} \left(\frac{N_n}{D_n}\right)^{\!\gamma} \,\,\proba\!\left[Y_n - y \in \left(\frac{D_n}{N_n}\right)^{\!\gamma}B\right] = \frac{\E^{-\frac{y^2}{2}}}{\sqrt{2\pi}}\, m(B).$$
\end{proposition}

\begin{example}
As a particular case, suppose that $S_n = \sum_{i=1}^{N_n} A_i$ is a sum of random variables with $\|A_i\|_\infty \leq A$ for all $i \leq N_n$, and such that there exists a \emph{dependency graph} $G = (\lle 1,N_n\rre,E_n)$ with the following property:
\begin{enumerate}
    \item We have $D_n = 1+\max_{i \in \lle 1,N_n\rre} \deg(i)$.
    \item If $I$ and $J$ are two disjoint sets of vertices of $\lle 1,N_n\rre$ without edge $e \in E_n$ connecting a vertex $i \in I$ with a vertex $j \in J$, then $(A_i)_{i \in I}$ and $(A_j)_{j \in J}$ are independent.
\end{enumerate}
Then, it was shown in \cite[Section 9]{FMN16} that $S_n$ has uniform bounds on cumulants with parameters $(D_n,N_n,A)$. Therefore, if $S_n$ is a sum of random variables with a sparse dependency graph, then $Y_n = (S_n - \esper[S_n])/\sqrt{\Var(S_n)}$ usually satisfies a local limit theorem which holds up to the scale $\sqrt{D_n/N_n}$. \bigskip

For instance, consider the graph subcount $I(H,G_n)$ of a motive $H$ in a random Erd\"os--R\'enyi graph $G_n$ of parameters $(n,p)$, with $p \in (0,1)$ fixed (see \cite[Example 4.10]{FMN17} for the precise definitions). It is shown in \emph{loc.~cit.} that $I(H,G_n)$ admits a dependency graph with parameters \begin{align*}
N_n &= n^{\downarrow k};\\
D_n &= 2\binom{k}{2}\,(n-2)^{\downarrow k-2},
\end{align*}
where $k$ is the number of vertices of the graph $H$, and $n^{\downarrow k}=n(n-1)\cdots (n-k+1)$. Moreover, \begin{align*}
\esper[I(H,G_n)]&=p^h n^{\downarrow k};\\
\Var(I(H,G_n)) &= 2h^2 p^{2h-1}(1-p)\,n^{2k-2} + O(n^{2k-3}),
\end{align*}
where $h$ is the number of edges of $H$. Therefore, we have the local limit theorem:
$$\proba\!\left[\frac{I(H,G_n)}{p^h}-n^{\downarrow k} - 2h n^{k-1}x \in n^{k-1-\gamma}\,B\right] \simeq n^{-\gamma}\,\frac{\E^{-\frac{px^2}{1-p}}\,m(B)}{2h\sqrt{\pi\,(\frac{1}{p}-1)}}$$
for any $\gamma \in (0,1)$. For example, if $T_n$ is the number of triangles in a random Erdös--Rényi graph $G(n,p)$, then for any $\gamma \in (0,1)$,
$$\lim_{n \to \infty} n^{\gamma}\,\,\proba\!\left[\frac{T_n}{p^3} - n^{\downarrow 3} -6n^2 x\in n^{2-\gamma}\,B\right]  = \frac{\E^{-\frac{px^2}{1-p}}\,m(B)}{6\sqrt{\pi\,(\frac{1}{p}-1)}}.$$
We cannot attain with our method the discrete scale (which would correspond to the exponent $\gamma=2$ in the case of triangles). In the specific case of triangles, this strong local limit theorem has been proved recently by Gilmer and  Kopparty, see \cite{GK16}. Our local limit theorem holds at larger scales and for any graph subcount.
\end{example}
\medskip

\subsection{Numbers of visits of a finite Markov chain}\label{subsec:markov}
The method of cumulants can also be applied to sums of random variables that are all dependent (there is no sparse dependency graph), but still with a ``weak'' dependency structure. We refer to \cite[Section 5]{FMN17}, where this is made rigorous by means of the notion of weighted dependency graph. Consider for instance an ergodic Markov chain $(X_n)_{n \in \N}$ on a finite state space $\mathfrak{X}=\lle 1,M\rre$, where by ergodic we mean that the transition matrix $P$ of $(X_n)_{n \in \N}$ is irreducible and aperiodic. We also fix a state $a \in \lle 1,M\rre$, and we denote $\pi(a)$ the value of the unique stationary measure $\pi$ of the Markov chain at $a$. If $T_a$ is the first return time to $a$, then it is well known that $\pi(a) = \frac{1}{\esper_a[T_a]}$. In the sequel, we assume to simplify that the Markov chain has for initial distribution the stationary measure $\pi$, and we denote $\proba$ and $\esper$ the corresponding probabilities and expectations on trajectories in $\mathfrak{X}^{\N}$. If
$$N_{n,a} = \card \{ i \in \lle 1,n\rre\,|\,X_i=a\}$$
is the number of visits of $a$ from time $1$ to time $n$, then $\esper[N_{n,a}] =n\,\pi(a)$ and 
$$\lim_{n \to \infty}\frac{\Var(N_{n,a})}{n} = (\pi(a))^3\,\Var(T_a).$$
This identity is a particular case of the following general result: if $f$ is a function on $\mathfrak{X}$, then
$$\lim_{n \to \infty} \frac{\Var(\sum_{i=1}^n f(X_i))}{n} = \pi(a)\,\esper_a\!\left[\left(\sum_{i=1}^{T_a} f(X_i)-\pi(f)\right)^2\right].$$
In \cite[Theorem 5.14]{FMN17}, we proved that $N_{n,a}$ has uniform bounds on cumulants with parameters $A=1$, $N_n=n$, and $D_n = \frac{1+\theta_P}{1-\theta_P}$, where $\theta_P <1$ is a constant depending only on $P$ (it is the square root of the second largest eigenvalue of the multiplicative reversiblization $P\widetilde{P}$ of $P$, see \cite[\S2.1]{Fill91} for details on this construction). From this, we deduce:
\begin{proposition}
Let $(X_n)_{n \in \N}$ be a stationary finite ergodic Markov chain, and $a$ be an element of the space of states. The numbers of visits $N_{n,a}$ satisfy the local limit theorem
$$\lim_{n \to \infty}\,n^{\gamma}\,\,\proba\!\!\left[\frac{N_{n,a} - n\pi(a)}{\sqrt{\Var(N_{n,a})}} - x \in n^{-\gamma}B\right] = \frac{\E^{-\frac{x^2}{2}}\,m(B)}{\sqrt{2\pi}}$$
for any $\gamma \in (0,\frac{1}{2})$ and any Jordan measurable subset $B$ with $m(B)>0$.
\end{proposition}
For finite ergodic Markov chains, the discrete local limit theorem $(\gamma = \frac{1}{2})$ is known and due to Kolmogorov, see 
\cite{Kol49}. However, since there is no uniformity in the local estimates of $\proba[N_{n,a}=k]$, our result is not a direct consequence (and does not imply) the local limit theorem of Kolmogorov.
\bigskip

\section{Examples from random matrix theory}\label{sec:matrix}
In this section, we examine examples stemming from or closely related to random matrix theory, and which exhibit mod-Gaussian behavior.

\subsection{Number of charge one particles in a two charge system}\label{subsec:charge}
Let $L,M$ be non-negative random integers and $n\in\N$ be a fixed natural number, such that $L+2M=2n$.
We consider the two charge ensembles proposed in \cite{RSX13} and \cite{SS14}, where the particles are located on the real line, respectively on the unit circle. These models can be considered as interpolations between the classical ensembles GOE and GSE, respectively COE and CSE from random matrix theory.
\medskip

\paragraph{\emph{The real line.}}
The system consists of $L$ particles with unit charge and $M$ particles with charge two, located on the real line at positions $\xi=(\xi_1,\ldots,\xi_L)$ and $\zeta=(\zeta_1,\ldots,\zeta_M)$. We denote by
$E_{L,M}$ the total potential energy of the state $(\xi,\zeta)$. For this model $E_{L,M}$ is the sum of the total interaction energy between particles and of an external harmonic oscillator potential:
\begin{align*}
E_{L,M} &= -\sum_{1\leq i<j\leq L }\log |\xi_i-\xi_j| - 2 \sum_{i=1}^L\sum_{j=1}^M \log |\xi_i-\zeta_j|- 4\sum_{1\leq i<j\leq M }\log |\zeta_i-\zeta_j|\\
& \qquad+\sum_{i=1}^L \frac{(\xi_i)^2}{2} + 2\sum_{j=1}^M \frac{(\zeta_j)^2}{2}.
\end{align*}
The ensemble has population vector $(L,M)$ with probability proportional to 
$$X^LZ_{L,M},$$
where $X\ge 0$ is a parameter called \emph{fugacity} and $Z_{L,M}$ is given by
\begin{equation*}
Z_{L,M} = \frac{1}{L! M!} \int_{\mathbb{R}^L}  \int_{\mathbb{R}^M} \E^{-E_{L,M} (\xi, \zeta)} d\mu^L(\xi) d\mu^M(\zeta), 
\end{equation*}
$\mu^L,\ \mu^M$ being the Lebesgue measures on $\R^L$ and $\R^M$ respectively.
 We denote by $Z_n(X)$ the total partition function of the system, that is
$$Z_n(X)=\sum_{L+2M=2n}X^L Z_{L,M}.$$
Let $L_n(\gamma)$ represent the number of charge one particles, in the scaling $X=\sqrt{2n\gamma}$ with $\gamma>0$. In this regime, the proportion of such particles is non-trivial. In \cite{RSX13}, the authors gave a representation of the total partition function $Z_n(X)$ as a product of generalized Laguerre polynomials with parameters $-\frac{1}{2}$. As a consequence, in \cite[Section 2]{DHR18}, it is shown that the normalized sequence $$\left(\frac{L_n(\gamma)-\mathbb{E}\left[L_n(\gamma)\right]}{n^{1/3}}\right)_{n\in\N}$$
is mod-Gaussian convergent on the complex plane with parameters $t_n=n^{1/3}\sigma_n^2(\gamma)$ and limiting function $\psi(z)=\exp (M(\gamma)\frac{z^3}{6})$ ; here $\sigma_n^2(\gamma)=\frac{\text{Var}\left(L_n(\gamma)\right)}{n}$ and 
$$M(\gamma) = \lim_{n \to \infty} \frac{\kappa^{(3)}(L_n(\gamma))}{n}.$$ 
The precise values of $\sigma^2(\gamma) = \lim_{n \to \infty} \sigma_n^2(\gamma)$ and $M(\gamma)$ are computed in \cite[Proposition 2.2]{DHR18}.  On the other hand, the mod-Gaussian convergence has a zone of control of order $O\left(t_n\right)$ and index $(3,3)$. A straightforward application of Theorem \ref{thm:locallimit1} shows that for any $\delta \in (0,\frac{1}{2})$,
\begin{equation*}
\proba\left[\frac{L_n(\gamma)-\esper[L_n(\gamma)]}{n^{1/2}\,\sigma_n(\gamma)}-x\in n^{-\delta}B\right]\simeq n^{-\delta}\,\frac{\E^{-\frac{x^2}{2}}}{\sqrt{2\pi}}\,m(B).
\end{equation*}
\medskip

\paragraph{\emph{The unit circle.}}
In the circular version of the above ensemble introduced in \cite{SS14}, the particles are located on the unite circle, instead of on the real line. Unlike the previous model, the total energy of the system is given by the interacting energy between the particles, and there is no external field contributing to it:
\begin{align*}
E_{L,M} &= -\sum_{1\leq i<j\leq L }\log |\xi_i-\xi_j| - 2 \sum_{i=1}^L\sum_{j=1}^M \log |\xi_i-\zeta_j|- 4\sum_{1\leq i<j\leq M }\log |\zeta_i-\zeta_j|.
\end{align*}
Let $L_n(\rho)$ denote the number of charge one particles, in the scaling $X=2n\rho$ with $\rho>0$. Using the polynomial product structure of the partition function established by Forrester (see \cite[Section 7.10]{For10}), it is possible to prove \cite[Section 3]{DHR18} that the normalized sequence $$\left(\frac{L_n(\rho)-\mathbb{E}\left[L_n(\rho)\right]}{n^{1/3}}\right)_{n\in\N}$$
converges mod-Gaussian on the whole complex plane with parameters $t_n=n^{1/3}\sigma_n^2(\rho)$, limiting function
 $$\psi(z)=\exp \left(\frac{z^3}{6} \left(\rho\arctan\frac{1}{\rho}-\frac{\rho^4+3\rho^2}{\left(\rho^2+1\right)^2}\right) \right)$$ 
 and with a zone of control of order $O\left(t_n\right)$ and index $(3,3)$. Again $\sigma_n^2(\rho)=\frac{\text{Var}\left(L_n(\rho)\right)}{n}$. Therefore, for any $\delta \in (0,\frac{1}{2})$,
\begin{equation*}
\proba\left[\frac{L_n(\rho)-\mathbb{E}\left[L_n(\rho)\right]}{n^{1/2}\,\sigma_n(\rho)}-x\in n^{-\delta}B\right]\simeq n^{-\delta}\,\frac{\E^{-\frac{x^2}{2}}}{\sqrt{2\pi}}\,m(B).
\end{equation*}
\medskip

\subsection{Determinant of the GUE}\label{subsec:gue}
Let $W_n^H$ be a $n\times n$ random matrix in the Gaussian Unitary Ensemble.
Denote by
\begin{equation*}
X_n^{H}:=\log|\det W_{n}^{H}|-\mu_n^H,
\end{equation*}
the logarithm of the modulus of the determinant, properly centered. The centering is given by
\begin{equation*}
\mu_n^H=\frac{1}{2}\log2\pi-\frac{n}{2}+\frac{n}{2}\log n,
\end{equation*}
and corresponds (up to constant terms) to the expectation of $\log\det |W_{n}^{H}|$. 
The Mellin transform of this statistics has been calculated explicitly by Metha and Normand in \cite{MN98}. Thus
$$
\esper\!\left[|\det W_{n}^{H}|^z\right]=2^{\frac{nz}{2}}\,\prod_{k=1}^n\frac{\Gamma\!\left(\frac{z+1}{2}+\lfloor \frac{k}{2}\rfloor\right)}{\Gamma\!\left(\frac{1}{2}+\lfloor \frac{k}{2}\rfloor\right)},
$$
is analytic for all $z\in\C$ with $\mathrm{Re}(z)>-1$. Recently, the same representation has been derived by Edelman and La Croix in \cite{CE15}, noticing that $|\det W_n^H|$ is distributed as the product of the singular values of the GUE. Relying on this explicit formula, it is possible to prove that the sequence $(X_n^H)_{n\in\N}$ is mod-Gaussian convergent on $D=\left(-1,+\infty\right)\times \I\R$, with parameters $t_n=\frac{1}{2}\log\frac{n}{2}$ and limiting function
\begin{equation*}
\psi(z)=\log\frac{\Gamma(\frac{1}{2})\,\left(G(\frac{1}{2})\right)^2}{\Gamma(\frac{z+1}{2})\,\left(G(\frac{z+1}{2})\right)^2}.
\end{equation*}
Moreover, this sequence has a zone of control of size $O(t_n)$ and index $(1,3)$. Hence, by Theorem \ref{thm:locallimit1}, we obtain that
\begin{equation*}
\proba\left[\frac{X_n^H}{\sqrt{\frac{1}{2}\log\frac{n}{2}}}-x\in (\log n)^{-\delta}B\right]\simeq (\log n)^{-\delta}\,\frac{\E^{-\frac{x^2}{2}}}{\sqrt{2\pi}}\,m(B),
\end{equation*}
for every $\delta\in \left(0,\frac{3}{2}\right)$.\medskip

\subsection{Determinants of \texorpdfstring{$\beta$}{beta}-ensembles}\label{subsec:beta}
A result analogous to the previous one can be established for log-determinants of matrices in some well-known $\beta$-ensembles. Namely, let $W_{n}^{i,\beta}$ be a random matrix in the:
\begin{enumerate}
\item[\footnotesize$(i=L)$] Laguerre ensemble with parameters $(n,n,\beta)$,
\item[\footnotesize$(i=J)$] Jacobi ensemble, with parameters $(\lfloor n\tau_1\rfloor,\lfloor n\tau_1\rfloor,\lfloor n\tau_2\rfloor,\beta)$, where $\tau_1,\,\tau_2>0$,
\item[\footnotesize$(i=G)$] Uniform Gram ensemble of parameters $(n,n,\beta)$.
\end{enumerate}
We refer to \cite[Section 3]{DHR19} for the precise definitions. For all $i$, denote by 
\begin{equation*}
X_n^{i,\beta}:=\log\det W_{n}^{i,\beta}-\mu_{n}^{i,\beta}
\end{equation*}
the logarithm of the determinant, properly centered. As for the GUE case, the centering parameters $\mu_{n}^{i,\beta}$ correspond (up to constant terms) to the expectation of the log-determinants; see \cite[Lemma 4.2]{DHR19} for their explicit expressions.\medskip

For these statistics, the moment generating functions can be calculated by means of Selberg integrals. As a consequence, they are all given by a product of Gamma functions; for instance, for the $\beta$-Laguerre ensemble,
$$\esper\left[\E^{zX_n^{L,\beta}}\right]=\E^{-z\mu_{n}^{L,\beta}}2^{nz}\prod_{k=1}^n\frac{\Gamma(\frac{\beta}{2}k+z)}{\Gamma(\frac{\beta}{2}k)}.
$$
Classical techniques of complex analysis enable us to find an asymptotic expansion of these product formulas as $n$ goes to infinity. In turn, these expansions imply mod-Gaussian convergence for all the sequences $(X_n^{i,\beta})_{n\in\N}$ on
$$D=\left(-\frac{\beta}{2},+\infty\right)\times \I\R,$$
with parameters $t_n=\frac{2}{\beta}\log n$
and a zone of control of size $O(t_n)$, with index $(1,3)$.\\
Therefore, for all $i=L,J,G$ and any $\delta\in\left(0,\frac{3}{2}\right)$,
\begin{equation*}
\proba\left[\frac{X_n^{i,\beta}}{\sqrt{\frac{2}{\beta}\log n}}-x\in (\log n)^{-\delta}B\right]\simeq (\log n)^{-\delta}\,\frac{\E^{-\frac{x^2}{2}}}{\sqrt{2\pi}}\,m(B).
\end{equation*}
\medskip

\subsection{Characteristic polynomial of the circular \texorpdfstring{$\beta$}{beta}-Jacobi ensemble} \label{subsec:circular}
Let $W_n^{CJ,\beta}$ be a random matrix in the circular $\beta$-Jacobi ensemble of size $n$. We recall that the joint density function of the eigenangles $\left(\theta_1,\ldots,\theta_n\right)\in [0,2\pi]^n$ is proportional to
$$\prod_{1\le k<j\le n}\left|\E^{i\theta_k}-\E^{i\theta_l}\right|\prod_{k=1}^n\left(1-\E^{-i\theta_k}\right)^\delta\left(1-\E^{i\theta_k}\right)^{\bar{\delta}},$$
with $\delta\in\C$, $\mathrm{Re} (\delta)>-\frac{1}{3}$. Denote by
$$X_n^{CJ,\beta}:=\log\det\left|\mathrm{Id}-W_n^{CJ,\beta}\right|- \frac{\delta+\bar{\delta}}{\beta}$$
the logarithm of the determinant of the characteristic polynomial evaluated at $1$ and properly shifted. 
Starting from the representation of Laplace transform of $X_n^{CJ,\beta}$ computed in \cite[Formula 4.2]{BNR09}, one can establish the complex mod-Gaussian convergence of the sequence 
 $(X_n^{CJ,\beta})_{n\in\N}$ on the subset
$$D=\left(-\frac{\beta}{2},+\infty\right)\times \I\R,$$
with parameters $t_n=\frac{1}{2\beta}\log n$  and a zone of control of size $O(t_n)$, with index $(1,3)$. It follows then from Theorem \ref{thm:locallimit1} that for all $\delta\in\left(0,\frac{3}{2}\right)$,
\begin{equation*}
\proba\left[\frac{X_n^{CJ,\beta}}{\sqrt{\frac{1}{2\beta}\log n}}-x\in (\log n)^{-\delta}\,B\right]\simeq (\log n)^{-\delta}\,\frac{\E^{-\frac{x^2}{2}}}{\sqrt{2\pi}}\,m(B).
\end{equation*}
\bigskip

\section{\texorpdfstring{$\leb^1$}{L1}-mod-\texorpdfstring{$\phi$}{phi} convergence and local limit theorems}\label{sec:l1mod}

\subsection{Mod-\texorpdfstring{$\phi$}{phi} convergence in \texorpdfstring{$\leb^1(\I \R)$}{L1(iR)}}\label{subsec:l1}
In all the previous examples, by using the notion of zone of control, we identified a range of scales $(t_n)^{-\delta}$ at which the stable approximation of the random variables $Y_n$ holds. However, in certain examples, one has a control over the residues $\theta_n(\xi)$ that is valid over the whole real line. This raises the question whether the theory of mod-$\phi$ convergence can be used to prove local limit theorems that hold \emph{for any} infinitesimal scale. A sufficient condition for these strong local limit theorems is the notion of mod-$\phi$ convergence in $\leb^1(\I \R)$. It is a more abstract and restrictive condition than the notion of zone of control used in Theorem \ref{thm:locallimit1}, but it yields stronger results.

\begin{definition}\label{def:L1modphi}
Fix a reference stable law $\phi=\phi_{c,\alpha,\beta}$.
Let $(X_n)_{n\in\N}$ be a sequence that is mod-$\phi$ convergent on $D=\I \R$, with parameters $(t_n)_{n \in \N}$ and limiting function $\theta(\xi)$. We say that there is mod-$\phi$ convergence in $\leb^1(\I \R)$ if:
\begin{itemize}
     \item $\theta$ and the functions $\theta_n(\xi) = \esper[\E^{\I \xi X_n}]\,\eta^{-t_n\eta_{c,\alpha,\beta}(\I \xi)}$ belong to $\leb^1(\R)$;
     \item the convergence
\begin{equation*}
\theta_n(\xi)\longrightarrow\theta(\xi)
\end{equation*}
takes place in $\leb^1(\R)$: $\|\theta_n -\theta\|_{\leb^1(\R)}\to 0$.
 \end{itemize}  
\end{definition}
 
Roughly speaking, mod-convergence in $\leb^1(\I\R)$ is equivalent to the assumption that $\gamma=+\infty$ in the zone of control. The following theorem makes this statement more precise.

\begin{theorem}\label{thm:locallimit2}
Let $(X_n)_{n \in \N}$ be a sequence that converges mod-$\phi_{c,\alpha,\beta}$ in $\leb^1(\I\R)$, with parameters $(t_n)_{n \in \N}$ and limiting function $\theta$. 
Let $x \in \R$ and $B$ be a fixed Jordan measurable subset with $m(B)>0$. Then, for any sequence $s_n \to +\infty$,
$$ \lim_{n \to \infty} s_n\,\,\proba\!\!\left[Y_n-x \in \frac{1}{s_n}\,B\right] = p_{c,\alpha,\beta}(x)\,m(B),$$
where $Y_n$ is obtained from $X_n$ as in Proposition \ref{prop:convlaw}.
\end{theorem}

\begin{proof}
For the same reasons as in the proof of Theorem \ref{thm:locallimit1}, it suffices to prove the estimate on test functions $g \in \testf_0(\R)$:
$$\lim_{n \to \infty} s_n\,\,\esper\!\left[g(s_n (Y_n-x))\right] = p_{c,\alpha,\beta}(x)\,\left(\int_{\R}g(y)\DD{y}\right).$$
By using Parseval's theorem and making the adequate changes of variables, we get
$$\esper\!\left[g(s_n (Y_n-x))\right] = \frac{1}{2\pi\,s_n}\,\int_{\R} \widehat{g}\!\left(\frac{\xi}{s_n}\right)\,\theta_n\!\left(-\frac{\xi}{(t_n)^{1/\alpha}}\right)\,\E^{\eta(-\I\xi)+\I x \xi}\DD{\xi}. $$
The function under the integral sign converges pointwise towards $\widehat{g}(0)\,\E^{\eta(-\I \xi)+\I x \xi}$, and this convergence actually occurs in $\leb^1(\R)$. Indeed,
$$
\int_{\R} |\E^{\eta(-\I \xi)+\I x \xi}|\,\left|\widehat{g}\!\left(\frac{\xi}{s_n}\right)\,\theta_n\!\left(-\frac{\xi}{(t_n)^{1/\alpha}}\right)-\widehat{g}(0)\right|\DD{\xi} \leq A+B+C
$$
with
\begin{align*}
A&= \int_{\R}|\E^{\eta(-\I \xi)}|\,\left|\widehat{g}\!\left(\frac{\xi}{s_n}\right)\right|\,\left|\theta_n\!\left(-\frac{\xi}{(t_n)^{1/\alpha}}\right)-\theta\!\left(-\frac{\xi}{(t_n)^{1/\alpha}}\right)\right|\DD{\xi};\\
B&=\int_{\R}|\E^{\eta(-\I \xi)}|\,\left|\widehat{g}\!\left(\frac{\xi}{s_n}\right)\right|\,\left|\theta\!\left(-\frac{\xi}{(t_n)^{1/\alpha}}\right)-1\right|\DD{\xi};\\
C&=\int_{\R}|\E^{\eta(-\I \xi)}|\,\left|\widehat{g}\!\left(\frac{\xi}{s_n}\right)-\widehat{g}(0)\right|\DD{\xi}.
\end{align*}
Since $\widehat{g}$ is bounded and $\widehat{g}(\frac{\xi}{s_n})-\widehat{g}(0)\to 0$ pointwise, one can apply the dominated convergence theorem to show that $C  \to 0$. For $A$, we make another change of variables and write
\begin{align*}
A&=\int_{\R}(t_n)^{1/\alpha}\,\E^{-t_n |c\upsilon|^\alpha}\,\left|\widehat{g}\!\left(\frac{\upsilon}{s_n\,(t_n)^{-1/\alpha}}\right)\right|\,\left|\theta_n(-\upsilon)-\theta(-\upsilon)\right|\,d\upsilon\\
&\leq \|\widehat{g}\|_\infty\,\int_{\R}(t_n)^{1/\alpha}\,\E^{-t_n |c\upsilon|^\alpha}\,\left|\theta_n(-\upsilon)-\theta(-\upsilon)\right|\,d\upsilon. 
\end{align*}
Fix $\eps>0$. Since $\theta_n$ converges locally uniformly towards $\theta$, there exists an interval $[-C,C]$ such that $|\theta_n(-\upsilon)-\theta(-\upsilon)| \leq \eps$ for any $\upsilon \in [-C,C]$. The part of $A$ corresponding to this interval is therefore smaller than
$$\eps\,\|\widehat{g}\|_\infty\,\int_{-C}^C(t_n)^{1/\alpha}\,\E^{-t_n |c\upsilon|^\alpha}\,d\upsilon \leq \eps \,\|\widehat{g}\|_\infty\,\int_{\R}\E^{ -|c\xi|^\alpha}\DD{\xi},$$
that is to say a constant times $\eps$. On the other hand, for $|\upsilon| \geq C$, 
$$(t_n)^{1/\alpha}\,\E^{-t_n |c\upsilon|^\alpha} \leq (t_n)^{1/\alpha}\,\E^{-t_n (cC)^\alpha} \to 0,$$
and the part of $A$ corresponding to $\R \setminus [-C,C]$ is smaller than
$$\|\widehat{g}\|_\infty\,(t_n)^{1/\alpha}\,\E^{-t_n (cC)^\alpha} \|\theta_n-\theta\|_{\mathrm{L}^1} \to 0 $$
since $\|\theta_n-\theta\|_{\leb^1}$ goes to zero. Hence, $A$ goes to zero. The same arguments allows one to show that $B \to 0$, using the continuity of $\theta$ at zero instead of the convergence $\theta_n \to \theta$ for the integral over an interval $[-C,C]$.\medskip

As a consequence of the convergence in $\leb^1$, one can now write
$$\esper\!\left[g(s_n (Y_n-x))\right] \simeq \frac{1}{2\pi\,s_n}\,\int_{\R} \widehat{g}(0)\,\E^{\eta(\I\xi)-\I x\xi}\DD{\xi} = \frac{1}{s_n}\,p_{c,\alpha,\beta}(x)\,\left(\int_{\R}g(y)\DD{y}\right),$$
which is what we wanted to prove.
\end{proof}
\medskip

\subsection{The winding number of the planar Brownian motion}\label{subsec:brownian}
As an application of our theory of $\leb^1$-mod-$\phi$ convergence, consider a standard planar Brownian motion $\left(Z_t\right)_{t\ge 0}$ starting at the point $(1,0)$. With probability 1, $Z_t$ does not visit the origin, so one can write $Z_t=R_t\, \E^{\I\varphi_t}$ with continuous functions $t\rightarrow R_t$ and $t\rightarrow \varphi_t$, and with $\varphi_0=0$. The process $\left(\varphi_t\right)_{t\ge 0}$ is called \emph{winding number} of the Brownian motion around the origin. Its Fourier transform has been calculated by Spitzer (see \cite{Spi58}), in terms of the modified Bessel function $I_\nu(z)=\sum_{k\ge 0}\frac{1}{k!\,\Gamma(\nu+k+1)}\left(\frac{z}{2}\right)^{\nu+2k}$. Thus,
\begin{equation*}
\esper\!\left[\E^{\I\xi\varphi_t}\right]=\sqrt{\frac{\pi}{8t}}\,\E^{-\frac{1}{4t}}\left(I_{\frac{|\xi|-1}{2}}\left(\frac{1}{4t}\right)+I_{\frac{|\xi|+1}{2}}\left(\frac{1}{4t}\right)\right).
\end{equation*}
As a consequence, in \cite{FMN17} it is shown that $\left(\varphi_t\right)_{t\ge 0}$ converges mod-Cauchy with parameters $\log \sqrt{8t}$, limiting function $$\theta(\xi)=\frac{\sqrt{\pi}}{\Gamma (\frac{|\xi|+1}{2})}$$
and with a control of index $(1,1)$ \emph{over the whole real line} $\R$. Since $\frac{1}{\omega-\alpha} = +\infty$, this means that one can consider zones of control $[-K(t_n)^\gamma,K(t_n)^\gamma]$ with $\gamma$ as large as wanted. In the sequel, we shall rework this example by using the notion of mod-$\phi$ convergence in $\leb^1(\I\R)$. 
\medskip

Notice first that
$$\Gamma\left(\frac{|\xi|+1}{2}\right) \geq \frac{2}{1+|\xi|}\,\,\Gamma\!\left(1+\frac{|\xi|}{2}\right) \geq \frac{2}{1+|\xi|}\,\left(\frac{|\xi|}{2\E}\right)^{\frac{|\xi|}{2}},$$
so that the limiting function $\theta(\xi)=\sqrt{\pi}\,(\Gamma (\frac{|\xi|+1}{2}))^{-1}$ is in $\leb^1(\R)$. So are the functions $\theta_t$. It remains to check that the convergence $\theta_t\rightarrow \theta$ happens in $\leb^1(\R)$: 
\begin{align*}
\|\theta_t-\theta\|_{\leb^1} &
\leq \left\|\sum_{k\ge 0}\left(\frac{\sqrt{\pi}\E^{-\frac{1}{4t}}}{k!\Gamma(k+\frac{|\xi|+1}{2})}\left(\frac{1}{8t}\right)^{2k}-\frac{\sqrt{\pi}\E^{-\frac{1}{4t}}}{k!\Gamma(k+\frac{|\xi|+3}{2})}\left(\frac{1}{8t}\right)^{2k+1}\right)-\frac{\sqrt{\pi}}{\Gamma(\frac{|\xi|+1}{2})}\right\|_{\leb^1}\\
&\leq \|\theta\|_{\leb^1}\left(1-\E^{-\frac{1}{4t}}\right)+ \sqrt{\pi}\,\E^{-\frac{1}{4t}} \left( \sum_{k=1}^\infty \frac{1}{k!}\left(\frac{1}{8t}\right)^{\!2k}\left\|\frac{1}{\Gamma(k+\frac{|\xi|+1}{2})}\right\|_{\leb^1} \right)\\
&\quad + \sqrt{\pi}\,\E^{-\frac{1}{4t}} \left( \sum_{k=0}^\infty \frac{1}{k!}\left(\frac{1}{8t}\right)^{\!2k+1}\left\|\frac{1}{\Gamma(k+\frac{|\xi|+3}{2})}\right\|_{\leb^1} \right)\\
&\leq \|\theta\|_{\leb^1}\left(\left(1-\E^{-\frac{1}{4t}}\right) + \E^{-\frac{1}{4t}} \left(\E^{\left(\frac{1}{8t}\right)^2}-1\right) + \frac{1}{8t}\, \E^{-\frac{1}{4t}+\left(\frac{1}{8t}\right)^2}\right) \stackrel{t\to \infty}{\longrightarrow} 0.
\end{align*}
Hence, $(\varphi_t)_{t \in \R_+}$ converges mod-Cauchy in $\leb^1(\I\R)$, and for any family $(s_t)_{t \in \R_+}$ growing to infinity, 
$$\lim_{t \to +\infty} s_t\,\,\proba\!\left[\frac{\varphi_t}{\log \sqrt{8t}} - x \in \frac{1}{s_t}\,B\right] = \frac{m(B)}{\pi(1+x^2)}.$$
\medskip

\subsection{The magnetisation of the Curie--Weiss model}\label{subsec:curieweiss}
In \cite{MN15}, another notion of mod-$\phi$ convergence in $\leb^1$ was introduced, in connection with models from statistical mechanics.
\begin{definition}\label{def:L1R}
Let $(X_n)_{n\in \N}$ be a sequence of random variables that is mod-Gaussian convergent on $D=\R$ (beware that the domain here is $\R$ and not $\I \R$), with parameters $(t_n)_{n \in \N}$ and limiting function $\psi(x)$. We say that there is mod-Gaussian convergence in $\leb^{1}(\R)$ if:
\begin{itemize}
     \item $\psi$ and the functions $\psi_n(x) = \esper[\E^{xX_n}]\,\E^{-\frac{t_nx^2}{2}}$ belong to $\leb^1(\R)$;
     \item the convergence
\begin{equation*}
\psi_n(x)\longrightarrow \psi(x)
\end{equation*}
occurs in $\leb^1(\R)$: $\|\psi_n -\psi\|_{\leb^1(\R)}\to 0$.
 \end{itemize}
\end{definition}

\noindent This definition mimics Definition \ref{def:L1modphi} with a domain $\R$ instead of $\I\R$.  This framework  allows one to prove the convergence in distribution for sequences which are obtained from $(X_n)_{n\in\N}$ by an exponential change of measure. We recall without proof this result (see \cite[Theorem 6]{MN15}).

\begin{theorem}\label{thm:changeofmeasure}
Let $(X_n)_{n \in \N}$ be a sequence of real-valued random variables that converges mod-Gaussian in $\leb^1(\R)$, with parameters $(t_n)_{n \in \N}$ and limit $\psi$. Denote by $(Y_n)_{n \in \N}$ the sequence obtained by the change of measures
\begin{equation*}
\proba_{Y_n}[\!\DD{x}] = \frac{\E^{\frac{x^2}{2t_n}} }{\esper\!\left[\E^{\frac{(X_n)^2}{2t_n}}\right]}\,\proba_{X_n}[\!\DD{x}].
\end{equation*}
Then $(\frac{Y_n}{t_n})_{n \in \N}$ converges in law towards a random variable $W_\infty$ with density $\frac{\psi(x)\DD{x}}{\int_{\R}\psi(x)\DD{x}}$.
\end{theorem}

In this setting and with mild additional hypotheses, one can identify the infinitesimal scales at which a local limit theorem holds for $(Y_n)_{n \in \N}$. The precise assumptions are the following:
\begin{assumption}
Let $(X_n)_{n \in \N}$ be a sequence of real-valued random variables. We assume that:
\begin{enumerate}[label=(A\arabic*)]
     \item\label{hyp:change1} The sequence $(X_n)_{n \in \N}$ is mod-Gaussian convergent in $\leb^1(\R)$, with parameters $(t_n)_{n \in \N}$ and limit $\psi$.
  \item\label{hyp:change3}
  For every $M>0$,
  $$ \sup_{n \in \N}\sup_{m \in [-M,M]}\left(\int_{\R} |\psi_n(x+\I m)|\DD{x} \right)<+\infty.$$
  We denote $C(M)$ the constant in this bound.
 \end{enumerate}
\end{assumption}

\begin{theorem}\label{thm:locallimitchange}
If Conditions \ref{hyp:change1} and \ref{hyp:change3} are satisfied, then for any $\eps \in (0,1]$,
$$\lim_{n \to \infty} (t_n)^{\eps} \,\,\proba\!\left[\frac{Y_n}{t_n}-x\in \frac{1}{(t_n)^{\eps}}\,B\right] = \frac{\psi(x)\,m(B)}{\int_{\R}\psi(y)\DD{y}},$$
where $Y_n$ is obtained from $X_n$ by the exponential change of measure of Theorem \ref{thm:changeofmeasure}.
\end{theorem}

\begin{lemma}\label{lem:exponentialdecay}
If $(X_n)_{n \in \N}$ satisfies Conditions \ref{hyp:change1} and \ref{hyp:change3}, then
$$|\widehat{\psi_n}(\xi)| \leq 2C(M)\,\E^{-M|\xi|}\quad\text{for any $\xi\in \R$ and any $M>0$}.$$
\end{lemma}

\begin{proof}
This is the content of \cite[p.~132]{RS75}, which we reproduce here for the convenience of the reader. Set $\psi_{n,M}(x)=\psi_n(x+\I M)$. Applying the Cauchy integral theorem,
\begin{align*}
\widehat{\psi_{n,M}}(\xi) &= \int_{\R}\psi_n(x+\I M)\,\E^{\I x \xi} \DD{x} =  \left(\int_{\R}\psi_n(x+\I M)\,\E^{\I (x + \I M) \xi} \DD{x}\right)\E^{M\xi} \\
&= \left(\int_{\R}\psi_n(x)\,\E^{\I x\xi} \DD{x}\right)\E^{M\xi} = \widehat{\psi_n}(\xi)\,\E^{M\xi}
\end{align*}
by analyticity of the function $\psi_n(z)\,\E^{\I z\xi}$, and existence and boundedness of all the integrals $ \int_{\R}\psi_n(x+\I m)\,\E^{\I (x+\I m)\xi}\DD{x}$. Therefore, 
\begin{equation*}
|\widehat{\psi_n}(\xi)|\,\E^{M|\xi|} \leq 
|\widehat{\psi_n}(\xi)|\left(\E^{M\xi}+\E^{-M\xi}\right)\leq
|\widehat{\psi_{n,M}}(\xi)| +|\widehat{\psi_{n,-M}}(\xi)| \leq 2C(M).\qedhere
\end{equation*}
\end{proof}
\medskip

\begin{proof}[Proof of Theorem \ref{thm:locallimitchange}] In the sequel we set $I_n := \int_{\R} \psi_n(x)\DD{x}$ and $I_\infty := \int_{\R} \psi(x) \DD{x}$. As usual, it is sufficient to prove the estimate with test functions $f \in \testf_0(\R)$:
$$\lim_{n \to \infty} (t_n)^{\eps} \,\,\esper\!\left[g\!\left((t_n)^{\eps}\,\left(\frac{Y_n}{t_n}-x\right)\right)\right] = \frac{\psi(x)}{I_\infty}\,\left(\int_{\R} g(y)\DD{y}\right).$$
We compute, with $\widehat{g}$ compactly supported on $[-M,M]$ and $g_n(y)=g((t_n)^{\eps}(y-x))$:
\begin{align*}
\esper\!\left[g\!\left((t_n)^{\eps}\,\left(\frac{Y_n}{t_n}-x\right)\right)\right] &= \esper\!\left[g_n\!\left(\frac{Y_n}{t_n}\right)\right] = \frac{1}{2\pi I_n} \int_{\R} \widehat{g_n}(\xi) \,\E^{\frac{\xi^2}{2t_n}}\, \widehat{\psi_n}(-\xi) \DD{\xi} \\
&= \frac{1}{2\pi I_n\,(t_n)^{\eps}} \int_{\R} \widehat{g}\left(\frac{\xi}{(t_n)^{\eps}}\right) \,\E^{\frac{\xi^2}{2t_n}+\I x \xi}\,\widehat{\psi_n}(-\xi)\DD{\xi}.
\end{align*}
In the integral, $\widehat{g}(\frac{\xi}{(t_n)^{\eps}}) \,\E^{\frac{\xi^2}{2t_n}+\I x \xi}\,\widehat{\psi_n}(-\xi)$ converges pointwise to $\widehat{g}(0)\,\E^{\I x \xi}\,\widehat{\psi}(-\xi)$. Moreover, it is dominated by 
$$2\,\|\widehat{g}\|_\infty\,C(M)\,\E^{-M\frac{|\xi|}{2}}
$$
by Lemma \ref{lem:exponentialdecay}, and using the fact that $\eps\leq 1$. Hence, by dominated convergence,
\begin{align*}
  \lim_{n \to \infty} (t_n)^{\eps} \,\esper\!\left[g\!\left((t_n)^{\eps}\,\left(\frac{Y_n}{t_n}-x\right)\right)\right] &= \frac{\widehat{g}(0)}{2\pi I_\infty} \left(\int_{\R} \E^{\I x \xi} \widehat{\psi}(-\xi)\DD{\xi}\right) \\
  &= \frac{\psi(x)}{I_\infty}\,\left(\int_{\R} g(y)\DD{y}\right).\qedhere
 \end{align*} 
\end{proof}

\begin{example}
The Curie--Weiss model at critical temperature is the probability law on spin configurations $\sigma=\left(\sigma_i\right)_{i\in \lle 1,n\rre}\in\{\pm 1\}^n$ given by
\begin{equation*}
\mathbb{CW}(\sigma)=\frac{\E^{\frac{1}{2n}\left(\sum_{i=1}^n\sigma_i\right)^2}}{\sum_{\sigma} \E^{\frac{1}{2n}\left(\sum_{i=1}^n\sigma_i\right)^2}}.
\end{equation*}
The random quantity $M_n:=\sum_{i=1}^n\sigma_i$ under the law $\mathbb{CW}$ is the \emph{total magnetization} of the Curie--Weiss model.\medskip

One can interpret $M_n$ in the mod-Gaussian convergence setting. Namely, consider be a sequence of i.i.d.~Bernoulli random variables $\left(\sigma_i\right)_{i\in\N}$ with $$\proba\left[\sigma_i=1\right]=1-\proba\left[\sigma_i=-1\right]=\frac{1}{2}$$
Then (see \cite[Theorem 8]{MN15}), $X_n=\frac{\sum_{i=1}^n \sigma_i}{n^{1/4}}$ is mod-Gaussian convergent in $\leb^1(\R)$ with parameters $t_n=\sqrt{n}$ and limiting function $\psi(x)=\exp(-\frac{x^4}{12})$. The sequence $Y_n$ obtained by the exponential change of measure of Theorem \ref{thm:changeofmeasure} has the same distribution as the rescaled total magnetization in the Curie--Weiss model, that is
$$Y_n\stackrel{\mathrm{law}}{=}\frac{M_n}{n^{1/4}}.$$ 
Note that, in this particular case,
$$\psi_n(\xi)=\E^{-\sqrt{n}\frac{x^2}{2}}\left(\cosh\left(\frac{x}{n^{1/4}}\right)\right)^n,\qquad \psi(x)=\E^{-\frac{x^4}{12}}.$$
From Proposition 15 in \cite{MN15}, we have that for all $M>0$
$$\sup_{n\in\N}\sup_{m\in[-M,M]}\left(\int_\R|\psi_n(x+\I m)| \DD{x}\right)\lesssim C(M)$$
with 
$$
C(M)=\E^{\frac{13M^4}{12}}\left( 2\sqrt{3}M+I_\infty\right),$$
and where $\lesssim$ means that the inequality holds up to a multiplicative constant $(1+\varepsilon)$, with $\varepsilon>0$, and for $n$ big enough.
Therefore, Conditions \ref{hyp:change1} and \ref{hyp:change3} are verified, and we can deduce from Theorem  \ref{thm:locallimitchange} the following local limit theorem: for any $\varepsilon\in (0,\frac{1}{2}]$, the following local limit theorem holds:
$$\lim_{n \to \infty} n^{\varepsilon}\,\,\proba\!\left[\frac{M_n}{n^{3/4}}-x\in \frac{B}{n^{\varepsilon}}\right] = \frac{\E^{-\frac{x^4}{12}}}{\int_{\R}\E^{-\frac{y^4}{12}}\DD{y}}\,m(B)$$
for any Jordan measurable subset $B$ with $m(B)>0$.
This improves on \cite[Theorem 22]{MN15}, which only dealt with the case $x=0$ and $\varepsilon = \frac{1}{2}$. In the same setting, one can also show that the Kolmogorov distance between $\frac{Y_n}{t_n}$ and its limit in law $W_\infty$ is a $O(\|\psi_n-\psi\|_{\leb^1(\R)}+\frac{1}{t_n})$, see Theorem 21 in \emph{loc.~cit.}
\end{example}

\bibliographystyle{alpha}
\bibliography{local_limit}

\end{document}